\documentclass[letter,11pt]{amsart}
\usepackage{multirow}
\usepackage{mathabx}
\usepackage{hyperref}
\usepackage{tikz}
\usetikzlibrary{decorations.markings}
\usetikzlibrary{shapes}
\usetikzlibrary{backgrounds}
\usepackage{subfig}
\usepackage{amssymb} 
\usepackage[all]{xy}

\def\thetitle{The geometry of the curve graph of a right-angled Artin group}
\def\theauthor{Sang-hyun Kim and Thomas Koberda}
\hypersetup{
    pdftitle=   \thetitle,
   pdfauthor=  {\theauthor}
}

\usepackage{enumerate}
\makeatletter
\let\@@enum@org\@@enum@
\def\@@enum@[#1]{\@@enum@org[\normalfont #1]}
\makeatother

\newcommand\form[1]{\langle #1\rangle}
\newcommand\syl[1]{\| #1 \|_{\mathrm{syl}}}
\newcommand\stl[1]{\| #1 \|_{\mathrm{*}}}
\def\opp{^\mathrm{opp}}
\newcommand\sgn{\mathrm{sgn}}
\newcommand\stab{\mathrm{Stab}}

\newcommand\co{\colon}

\newcommand\Z{\mathbb{Z}}
\newcommand\cay{\operatorname{Cayley}}
\newcommand\aut{\operatorname{Aut}}
\newcommand\Mono{\operatorname{Mono}}
\newcommand\PMod{\operatorname{PMod}}

\newcommand\Comm{\operatorname{Comm}}
\newcommand\Fill{\operatorname{Fill}}
\newcommand\Out{\operatorname{Out}}
\newcommand\supp{\operatorname{supp}}

\newcommand\lk{\operatorname{Lk}}
\newcommand\st{\operatorname{St}}
\newcommand\diam{\operatorname{diam}}

\newcommand\gex{{{\Gamma}^e}}
\newcommand\aga{{A(\Gamma)}}

\newcommand\Gam{\Gamma}
\newcommand\bZ{\mathbb{Z}}
\newcommand\Mod{\operatorname{Mod}}

\newcommand\mC{\mathcal{C}}
\newcommand\mL{\mathcal{L}}

\newtheorem{thm}{Theorem}
\newtheorem{lem}[thm]{Lemma}
\newtheorem{cor}[thm]{Corollary}
\newtheorem{prop}[thm]{Proposition}

\theoremstyle{rem}
\newtheorem{exmp}[thm]{Example}
\newtheorem*{rem}{Remark}

\theoremstyle{defn}
\newtheorem{defn}[thm]{Definition}

\def\be{\begin{enumerate}}
\def\ee{\end{enumerate}}

\begin{document}

\markboth{\theauthor}{\thetitle}

\title{\thetitle}

\author{Sang-hyun Kim}
\address{Department of Mathematical Sciences, Seoul National University, Seoul 151-747, Korea} \email{s.kim@snu.ac.kr}

\author{Thomas Koberda}
\address{Department of Mathematics, Yale University, 20 Hillhouse Ave, New Haven, CT 06520, USA}\email{thomas.koberda@gmail.com}

\date{\today}

\maketitle

\begin{abstract}
We develop an analogy between right-angled Artin groups and mapping class groups through the geometry of their actions on the extension graph and the curve graph respectively.  The central result in this paper is the fact that each right-angled Artin group acts acylindrically on its extension graph.  From this result we are able to develop a Nielsen--Thurston classification for elements in the right-angled Artin group. Our analogy spans both the algebra regarding subgroups of right-angled Artin groups and mapping class groups, as well as the geometry of the extension graph and the curve graph. On the geometric side, we establish an analogue of Masur and Minsky's Bounded Geodesic Image Theorem and their distance formula.
\end{abstract}

\section{Introduction}
\subsection{Overview}
In this article, we study the geometry of the action of a right-angled Artin group $\aga$ on its extension graph $\gex$.  The philosophy guiding this paper is that a right-angled Artin group $\aga$ behaves very much like the mapping class group $\Mod(S)$ of a hyperbolic surface $S$ from the perspective of the geometry of the action of $\aga$ on $\gex$, compared with the action of $\Mod(S)$ on the curve graph $\mC(S)$.  
The analogy between right-angled Artin groups and extension graphs versus mapping class groups and curve graphs 
is not perfect and it notably breaks down in several points, though it does help guide us to new results.

The results we establish in this paper can be divided into algebraic results and geometric results.  From the algebraic point of view, we discuss the role of the extension graph and of the curve graph in understanding the subgroup structure of right-angled Artin groups and mapping class groups respectively.  From the geometric point of view, we discuss not only the intrinsic geometry of the extension graph and curve graph, but also the geometry of the canonical actions of the right-angled Artin group and the mapping class group respectively.

The central observation of this paper is that from the point of view of coarse geometry, the extension graph can be thought of as the Cayley graph of the right-angled Artin group equipped with \emph{star length} rather than with word length.  Roughly speaking, the star length of an element $w\in\aga$ is the smallest $k$ for which $w=u_1\cdots u_k$, where each $u_i$ is contained in the subgroup generated by the star of a vertex in $\Gamma$.

Inspired by an analogous fact relating word length in the mapping class group with distance in the curve graph, we are able to refine distance estimates in $\gex$ by developing a theory of subsurface projections and proving a distance formula which recovers the \emph{syllable length} in $\aga$, at least for graphs of girth greater than four.  Roughly, the syllable length of $w\in\aga$ is the smallest $k$ for which $w=v_1^{n_1}\cdots v_k^{n_k}$, where each $v_i$ is a vertex in $\Gamma$ and $n_i\in\mathbb{Z}$.

In light of the preceding remarks, an alternative title for this article could be ``The geometry of the star metric on right-angled Artin groups".
In the interest of clarity and brevity, we will not state the results one by one here in the introduction.  In the next subsection we have included a tabular summary of the results in this paper, together with references directing the reader to the discussion of the corresponding result.

\subsection{Summary of results}
The following two tables summarize the main results of this article.  The results are recorded with parallel results in mapping class group theory in order to emphasize the analogy between the two objects.  For each result, either a reference will be given or the reader will be directed to the appropriate statement in this article.

\begin{table}
\begin{center}
\begin{tabular}{ | p{6.1cm}| p{6.1cm}|}
\hline
\multicolumn{2}{|c|}{{\bf Summary of Geometric Results}} \\
\hline
$\aga$ & $\Mod(S)$\\ \hline
{\bf Extension graph} $\gex$ & {\bf Curve graph} $\mC(S)$\\ \hline
$\gex$ is quasi--isometric to an electrification of $\cay(\aga)$ (Theorem \ref{thm:qi}) & $\mC(S)$ is quasi--isometric to an electrification of $\cay(\Mod(S))$ (\cite{masurminsky1})\\ \hline
Extension graphs fall into exactly two quasi--isometry classes (Theorem \ref{thm:starqiclass}) & Curve graphs are quasi--isometrically rigid (\cite{rafischleimer})\\ \hline
$\gex$ is a quasi--tree (\cite{KK2012}) & $\mC(S)$ is $\delta$--hyperbolic (\cite{masurminsky1})\\ \hline
The action of $\aga$ on $\gex$ is acylindrical (Theorem \ref{thm:acyl}) & The action of $\Mod(S)$ on $\mC(S)$ is acylindrical (\cite{bowditchtight})\\ \hline
Loxodromic--elliptic dichotomy for nonidentity elements (Section \ref{s:ntclass}) & Nielsen--Thurston classification (\cite{flp})\\ \hline
Each loxodromic element has a unique pair of fixed points on $\partial\gex$ (Lemma \ref{l:northsouth}) & Each pseudo-Anosov has a unique pair of fixed points on $\partial\mC(S)$ (\cite{flp})\\ \hline
Vertex link projection (Section \ref{s:proj}) & Subsurface projection (\cite{masurminsky1}, \cite{masurminsky2})\\ \hline
Bounded Geodesic Image Theorem for graphs with girth $\geq 5$ (Theorem \ref{thm:bgit}) & Bounded Geodesic Image Theorem (\cite{masurminsky2})\\ \hline
Distance formula coarsely measures syllable length in $\aga$ for graphs with girth $\geq 5$ (Section \ref{s:distance}) & Non--annular distance formula coarsely measures Weil--Petersson distance in Teichm\"uller space (\cite{masurminsky2} and \cite{brockconvexcores})\\ 
\hline
\end{tabular}
\caption{Main results, part one}
\end{center}
\end{table}

\begin{table}
\begin{center}
\begin{tabular}{ | p{6.1cm}| p{6.1cm}|}
\hline
\multicolumn{2}{|c|}{{\bf Summary of Algebraic Results}} \\
\hline
$\aga$ & $\Mod(S)$\\ \hline
{\bf Extension graph} $\gex$ & {\bf Curve graph} $\mC(S)$\\ \hline
Induced subgraphs of $\gex$ give rise to right-angled Artin groups of $\aga$ (Section \ref{s:algebra} and \cite{KK2012}) & Induced subgraphs of $\mC(S)$ give rise to right-angled Artin subgroups of $\Mod(S)$ (Subsection \ref{s:algebra} and \cite{Koberda2012})\\ \hline
An embedding $A(\Lambda)\to \aga$ gives rise to an embedding $\Lambda\to K(\gex)$ (Section \ref{s:algebra} and \cite{KK2012}) & An embedding $A(\Lambda)\to \Mod(S)$ gives rise to an embedding $\Lambda\to K(\mC(S))$ (Section \ref{s:algebra} and \cite{kkraagemb})\\ \hline
$\gex$ can be recovered from the intrinsic algebra of $\aga$ (Section \ref{s:alg}) & $\mC(S)$ can be recovered from the intrinsic algebra of $\Mod(S)$ (Section \ref{s:alg})\\ \hline
Cyclically reduced elliptic elements of $\aga$ are supported in joins (Theorem \ref{t:ntclass}) & Reducible mapping classes stablize sub--curve graphs (\cite{BLM1983}) \\ \hline
Injective homomorphisms from right-angled Artin groups to right-angled Artin groups and to mapping class groups preserve elliptics but not loxodromics
(Section \ref{s:type}) & 
Injective homomorphisms from mapping class groups to right-angled Artin groups and to mapping class groups preserve elliptics but not loxodromics 
(\cite{aramayonaleiningersouto})\\ \hline
Powers of loxodromic elements generate free groups (Theorem \ref{t:loxfree}) & Powers of pseudo-Anosov elements generate free groups (Proposition \ref{p:pseudoanosovfree})\\  \hline
Purely loxodromic subgroups are free (Theorem \ref{t:loxpure}) & One--ended purely pseudo-Anosov subgroups fall in finitely many conjugacy classes per isomorphism type (\cite{bowditchoneended})\\  \hline
Powers of pure elements generate right-angled Artin groups (Theorem \ref{thm:purepower}) & Powers of mapping classes with connected supports generate right-angled Artin groups (\cite{Koberda2012})\\ \hline
Automorphism group of $\gex$ is uncountable (Theorem \ref{thm:aut}) & Automorphism group of $\mC(S)$ is $\Mod(S)$ (\cite{ivanov})\\ \hline
\hline
\end{tabular}
\caption{Main results, part two}
\end{center}
\end{table}

\section{Preliminaries}
\subsection{Graph--theoretic terminology}
Throughout this paper, a \emph{graph} will mean a one-dimensional simplicial complex. In particular, graphs have neither loops nor multi--edges.
If there is a group action on a graph, we will assume that the action is a right-action.

Let $X$ be a graph. The vertex set and the edge set of $X$ are denoted by $V(X)$ and $E(X)$, respectively. 
We let ${V(X) \choose 2}$ denote the set of two-element subsets in $V(X)$,
and regard $E(X)$ as a subset of ${V(X)\choose 2}$.
We define the \emph{opposite graph} $X\opp$ of $X$ by the relations $V(X\opp)=V(X)$ and $E(X\opp) = {V(X) \choose 2}\setminus E(X)$.
For two graphs $X$ and $Y$, the \emph{join of $X$ and $Y$} is defined as the graph $X\ast Y = (X\opp\coprod Y\opp)\opp$.
A graph is called a \emph{join} if it is the join of two nonempty graphs. 
A subgraph which is a join is called a \emph{subjoin}.

For $S\subseteq V(X)$, the \emph{subgraph of $X$ induced by $S$} is a subgraph $Y$ of $X$ defined by the relations $V(Y)=S$ and 
\[E(Y) = \{e\in E(X)\;|\; \text{the endpoints of }e\text{ are in }S\}.\]
In this case, we also say $Y$ is an \emph{induced subgraph of $X$} and write $Y\le X$.
For two graphs $X$ and $Y$, we say that $X$ is \emph{$Y$--free} if no induced subgraphs of $X$ are isomorphic to $Y$. In particular, we say $X$ is \emph{triangle--free} (\emph{square--free}, respectively) if no induced subgraphs of $X$ are triangles (squares, respectively).

We say that $A\subseteq V(X)$ is a \emph{clique} in $X$ if every pair of vertices in $A$ are adjacent in $X$. The \emph{link} of a vertex $v$ in $X$ is the set of the vertices in $X$ which are adjacent to $v$, and denoted as $\lk(v)$. The \emph{star} of $v$ is the union of $\lk(v)$ and $\{v\}$, and denoted as $\st(v)$.
By a clique, a link or a star, we often also mean the subgraphs induced by them.
Unless specified otherwise, each edge of a graph is considered to have length one.
For a metric graph $X$, the distance between two points in $X$ is denoted as $d_X$, or simply by $d$ when no confusion can arise.

The \emph{girth} of a graph $\Gam$ is the length of the shortest cycle in $\Gam$.  By convention, the girth of a tree is infinite.

\subsection{Extension graphs}\label{ss:ext}
Let $G$ be a group and $A\subseteq G$. The \emph{commutation graph of $A$} is the graph having the vertex set $A$ such that two vertices are adjacent if the corresponding group elements commute. If $A$ is a set of cyclic subgroups of $G$, the commutation graph of $A$ will mean the commutation graph of 
the set $\{x_\alpha\co \alpha\in A\}$ where for each $\alpha\in A$ we choose a generator $x_\alpha$ for $\alpha$.

Suppose $\Gam$ is a finite graph. The \emph{right-angled Artin group on $\Gam$} is the group presentation 
\[ A(\Gam) = \form{ V(\Gam) \;|\; [a,b]=1\text{ for each }\{a,b\}\in E(\Gam)}.\]  We will refer to the elements of $V(\Gam)$ as the \emph{vertex generators} of $\aga$.
In~\cite{KK2012}, the authors defined the \emph{extension graph} $\Gam^e$ as the commutation graph of the vertex-conjugates in $A(\Gam)$. More precisely, the vertex set of $\Gam^e$ is $\{v^g\co v\in V(\Gam),g\in A(\Gam)\}$ and two distinct vertices $u^g$ and $v^h$ are adjacent if and only if they commute in $A(\Gam)$.
There is a natural right--conjugation action of $A(\Gamma)$ on $\Gamma^e$ defined by 
$v^h\mapsto v^{hg}$ for $v\in V(\Gamma)$ and $g,h\in A(\Gamma)$.

Observe that we may write \[\Gam^e=\bigcup_{g\in\aga}\Gam^g,\] where the notation $\Gam^g$ denotes the graph $\Gam$ with its vertices (treated as elements of $\aga$) replaced with their corresponding conjugates by $g$.  The adjacency relations in $\Gam^g$ are the same as in $\Gam$.  To obtain $\Gam^e$ from the set of conjugates $\{\Gam^g\mid g\in\aga\}$, we simply identify two vertices if they are equal, and similarly with two edges.

\subsection{Curve graphs}
Let $S=S_{g,n}$ be a connected, orientable surface of finite genus $g$ and with $n$ punctures.  We will assume that $2g+n-2>0$, so that $S$ admits a complete hyperbolic metric of finite volume.  We denote the \emph{mapping class group} of $S$ by $\Mod(S)$.  Recall that this group is defined to be the group of isotopy classes of orientation--preserving homeomorphisms of $S$.

By a \emph{simple closed curve} on $S$, we mean the isotopy class of an essential (which is to say nontrivial and non-peripheral in $\pi_1(S)$) closed curve on $S$ which has a representative with no self--intersections.
Observe that the three--times punctured sphere $S_{0,3}$ admits no simple closed curves.
For each simple closed curve $\alpha$, we denote by $T_\alpha$ the Dehn twist along $\alpha$.

Let $S\notin\{S_{0,3}, S_{0,4}, S_{1,1}\}$.  We define the \emph{curve graph} $\mC (S)$ of $S$ as follows: the vertices of $\mC (S)$ are simple closed curves on $S$, and two (distinct) simple closed curves are adjacent in $\mC (S)$ if they can be disjointly realized.  In other words, two isotopy classes $[\gamma_1]$ and $[\gamma_2]$ are connected by an edge if there exist disjoint representatives in those isotopy classes.  Thus, the curve graph of $S$ can be thought of the commutation graph of the set of Dehn twists in the mapping class group $\Mod(S)$.  The reader may recognize the curve graph as the $1$--skeleton of the curve complex of $S$.

The curve graph of $S$ needs to be defined differently in the case $S\in\{S_{0,3}, S_{0,4}, S_{1,1}\}$.  When $S=S_{0,3}$, we define $\mC(S)$ to be empty.  In the other two cases, observe that no two simple closed curves can be disjointly realized.  In these cases, we define two simple closed curves to be adjacent in $\mC(S)$ if they have representatives which intersect a minimal number of times.  Note that for $S_{0,4}$ this means two intersections, and for $S_{1,1}$ this means one intersection.

We will not be using any properties of curve graphs in the proofs of our results in this paper.  They will mostly serve to guide our intuition about extension graphs.

\subsection{Right-angled Artin subgroups}\label{s:algebra}
The goal of this subsection is to note that $\gex$ and $\mC(S)$ classify right-angled Artin subgroups of right-angled Artin groups and mapping class groups respectively, and that they do so in essentially the same way.  The reader will be directed to the appropriate references for proofs.

For a possibly infinite graph $X$, we define the graph $K(X)$ as follows (see~\cite{KK2012} and also \cite{kkraagemb}, where $K(X)$ is denoted as $X_k$).
The vertices of $K(X)$ are in bijective correspondence with the nonempty cliques of $X$. Two vertices $v_J$ and $v_L$ corresponding to cliques $J$ and $L$ are adjacent if $J\cup L$ is also a clique.
Note that $K(\mC(S))$ can be regarded as a \emph{multi-curve graph} of $S$ in the sense that each vertex corresponds to an isotopy class of a multi-curve consisting of pairwise non-isotopic loops
and two distinct multi-curves are adjacent if they do not intersect.
For two groups $H$ and $G$, we write $H\le G$ if there is an embedding from $H$ into $G$.

\begin{thm}[\cite{KK2012}]
Let $\Lambda$ and $\Gamma$ be finite graphs.
\begin{enumerate}
\item
If $\Lambda\le \gex$, then $A(\Lambda)\le \aga$.
More precisely, suppose $\phi$ is an embedding of $\Lambda$ into $\gex$ as an induced subgraph.
Then the map 
\[\phi_N:A(\Lambda)\to \aga\] defined by \[v\mapsto \phi(v)^N\] is injective for sufficiently large $N$.
\item
If $A(\Lambda)\le \aga$, then there exists an embedding from $\Lambda$ into $K(\gex)$
as an induced subgraph.
\end{enumerate}
\end{thm}

The corresponding result for mapping class groups is the following:

\begin{thm}[\cite{kkraagemb} and \cite{Koberda2012}]
Let $\Lambda$ be a finite graph and $S=S_{g,n}$ where $2g+n-2>0$.
\begin{enumerate}
\item
If $\Lambda\le \mC(S)$, then $A(\Lambda)\le \Mod(S)$.
More precisely, suppose $\phi$ is an embedding of $\Lambda$ into $\mC(S)$ as an induced subgraph.
Then the map 
\[\phi_N:A(\Lambda)\to\Mod(S)\] defined by \[v\mapsto T_{\phi(v)}^N\] is injective for sufficiently large $N$.
\item
If $A(\Lambda)\le\Mod(S)$, then there exists an embedding from $\Lambda$ into $K(\mC(S))$
as an induced subgraph.
\end{enumerate}
\end{thm}

\section{Intrinsic algebraic characterization of $\mC(S)$ and $\Gam^e$}\label{s:alg}
\subsection{Maximal cyclic subgroups}\label{ss:maximal}
In this section we would like to show that the intrinsic algebraic structure of a mapping class group $\Mod(S)$ and of a right-angled Artin group $\aga$ is sufficient to recover the curve graph $\mC(S)$ and the extension graph $\gex$ respectively.

Recall that the mapping class group has a finite index subgroup $\PMod(S)$, a \emph{pure mapping class group}, which consists of mapping classes $\psi$ such that if 
 $\psi$ stabilizes a multicurve $C$ then $\psi$ stabilizes $C$ component--wise and restricts to the identity or to a pseudo-Anosov mapping class on each component of $S\setminus C$.

\begin{lem}\label{lem:dtchar}
Let $G\le \PMod(S)$ be a cyclic subgroup which satisfies the following conditions:
\begin{enumerate}
\item
The centralizer of $G$ in $\PMod(S)$ contains a maximal rank abelian subgroup (among all abelian subgroups of $\PMod(S)$).
\item
There exists two maximal rank abelian subgroups $A,A'$ in the centralizer of $G$ such that $A\cap A'$ is cyclic and contains $G$ with finite index.
\end{enumerate}
Then there is a simple closed curve $c\subseteq S$ and a nonzero $k\in\bZ$ such that $G=\langle T_c^k\rangle$.
\end{lem}
\begin{proof}
Conditions (1) and (2) on $G$ together guarantee that a generator of $G$ is supported on a maximal multicurve on $S$.  Condition (2) guarantees that there are two maximal multicurves $C_1,C_2$ on $S$ which contain the support of $G$ and whose intersection $C_1\cap C_2$ consists of exactly one curve.
\end{proof}

\begin{prop}
Let $T$ be the set of the maximal cyclic subgroups of $\PMod(S)$ satisfying the conditions of Lemma \ref{lem:dtchar}.
Then $\mC(S)$ is isomorphic to the commutation graph of $T$.
\end{prop}
\begin{proof}
Define a map $\mC(S)$ to the commutation graph of $T$
by \[\phi:c\mapsto \langle T_c^k\rangle,\] where $k=k(c)$ is the smallest positive integer for which $T_c^k\in\PMod(S)$.  Such a $k$ exists since $\PMod(S)$ has finite index in $\Mod(S)$.  Since distinct isotopy classes of curves give rise to distinct Dehn twists and since two Dehn twists commute if and only the corresponding curves are disjoint, the map $\phi$ is well--defined.  If two Dehn twists do not commute then they generate a group which is virtually a nonabelian free group, so that the map $\phi$ preserves non--adjacency as well as adjacency.  By Lemma \ref{lem:dtchar}, $\phi^{-1}$ is defined and is surjective.  Thus $\phi$ is an isomorphism.
\end{proof}

In order to get an analogous result for right-angled Artin groups, we need to put some restrictions on $\Gamma$.  The reason for this is that a vertex generator (or its conjugacy class, more precisely) is not well--defined.  This is because a general right-angled Artin group $\aga$ has a very large automorphism group, and automorphisms may not preserve the conjugacy classes of vertex generators.

\begin{lem}\label{lem:vgchar}
Let $\Gamma$ be a connected, triangle-- and square--free graph.  Let $1\neq g\in\aga$ be a cyclically reduced element whose centralizer in $\aga$ is nonabelian.  Then there exists a vertex $v\in\Gamma$ and a nonzero $k\in\bZ$ such that $g=v^k$.
\end{lem}

\begin{proof}
Let $g$ satisfy the hypotheses of the lemma.  By the Centralizer Theorem (see~\cite{Servatius1989} and \cite[Lemma 5.1]{BC2010}), we have that $\supp(g)$ is contained in a subjoin of $\Gamma$, and that the full centralizer of $g$ is also supported on a subjoin of $\Gamma$.  Because $\Gamma$ has no triangles and no squares, every subjoin of $\Gamma$ is contained in the star of a vertex of $\Gamma$.  So, we may write $g=v^k\cdot g'$, where $v$ is a vertex of $\Gamma$, where $k\in\Z\setminus \{0\}$, and where $\supp(g')\subseteq\lk(v)$.  Observe that if $g'$ is not the identity then the centralizer of $g'$ in $\langle \st(v)\rangle$ is abelian.  It follows that $g=v^k$.
\end{proof}

\begin{defn}\label{d:abstract extension}
Let $G$ be a group and $T$ be the set of maximal cyclic subgroups of $G$ which have nonabelian centralizers.
Then the \emph{abstract extension graph $G^e$ of $G$} is defined as the commutation graph of $T$.
\end{defn}

Observe that if $\Gamma$ is a triangle--free graph without any degree--one or degree-zero vertex,
then the centralizer of each vertex is nonabelian. From this we can characterize powers of vertex--conjugates as follows.

\begin{prop}\label{p:raagrec}
Suppose $\Gamma$ is a finite, connected, triangle-- and square--free graph without any degree--one or degree--zero vertex.
Then for each finite-index subgroup $G$ of $A(\Gam)$, we have $G^e \cong \gex$.
\end{prop}
\begin{proof}
For each vertex $v\in V(\Gamma)$ and $g\in A(\Gamma)$, we let $n(v,g)=\inf\{n>0\co (v^g)^n\in G\}$.
Put $A = \{(v^g)^{n(v,g)}\co v\in V(\Gamma),g\in A(\Gamma)\}\subseteq A(\Gamma)$.
By Lemma~\ref{lem:vgchar}, we see that $G^e$ is the commutation graph of $A$.
It is immediate that $\phi\co \gex\to G^e$ defined by $\phi(v^g)  = (v^g)^{n(v,g)}$ is a graph isomorphism.
\end{proof}

\subsection{Consequences for commensurability}
A general commensurability classification for right-angled Artin groups is currently unknown.  The discussion in the previous subsection allows us to establish some connections between a commensurable right-angled Artin groups and their extension graphs.  
The following is almost immediate from the discussion in the preceding subsection, combined with the fact that 
if $G\le \aga$ is a nonabelian subgroup and $G'\le G$ has finite index, then $G'$ is also nonabelian:

\begin{cor}
Let $\Gamma$ and $\Lambda$ be connected, triangle-- and square--free graphs, and suppose that neither $\Gamma$ nor $\Lambda$ has any degree one vertices.  If $\aga$ is commensurable with $A(\Lambda)$ then $\gex\cong\Lambda^e$.
\end{cor}
\begin{proof}
If $\aga$ and $A(\Lambda)$ are abelian then they must both be cyclic, in which case the conclusion is immediate.  Otherwise, we can just apply Proposition \ref{p:raagrec} to suitable finite index subgroups of $\aga$ and $A(\Lambda)$.
\end{proof}

\begin{exmp}
Let $C_n$ denote the cycle on $n$ vertices.
Since the girths of $C_n^e$ is $n$, we see that $A(C_m)$ is not commensurable to $A(C_n)$ for $m\ne n\ge 3$.
\end{exmp}

One could analogously conclude that if $\Mod(S)$ and $\Mod(S')$ are commensurable then $\mC(S)\cong\mC(S')$.  In fact, if $\Mod(S)$ and $\Mod(S')$ are even quasi--isometric to each other then $S=S'$, by the quasi--isometric rigidity of mapping class groups (see \cite{BKMM2012} and \cite{hamrigidity}).

\section{Electrified Cayley graphs}\label{s:elec}
In~\cite{masurminsky1}, Masur and Minsky proved that an \emph{electrified Cayley graph} (defined below) of $\Mod(S)$ is quasi-isometric to the curve graph. Here, this result will be placed on a more general setting and applied to the action of right-angled Artin groups on extension graphs.

\subsection{General setting}
Let $G$ be a group with a finite generating set $\Sigma$ and let $Y=\cay(G,\Sigma)$.  
For convention, we assume $\Sigma=\Sigma^{-1}$.
The Cayley graph carries a natural metric $d_Y$.
Suppose $G$ acts simplicially and cocompactly on a graph $X$.  To avoid certain technical problems, we will assume that $X$ is connected.  We will write $d_X$ for the graph metric on $X$.
Let $A$ be a (finite) set of representatives of vertex orbits in $X$.

Put $H_\alpha = \stab_G(\alpha)$ for $\alpha\in A$.
We let $\hat{Y}$ be the \emph{electrification} of $Y$ with respect to the disjoint union of right-cosets $\coprod_{\alpha\in A} H_\alpha\setminus G$~\cite{farb1998}. This means that $\hat{Y}$ is the graph with $V(\hat{Y})=(\coprod_{\alpha\in A} H_\alpha\setminus G)\coprod G$ and $E(\hat{Y}) = E(Y) \coprod \{ \{g,C\}\co g\in C\in \coprod_{\alpha\in A}H_\alpha\setminus G\}$. Note our convention that even when $H_\alpha=H_\beta$, we distinguish the cosets of $H_\alpha$ an $H_\beta$ as long as $\alpha\ne\beta$.
The graph $\hat{Y}$ carries a metric $\hat{d}$ such that the edges from $Y$ have length $1$ and the other edges have length $1/2$. Let $T = \cup_{\alpha\in A} H_\alpha\cup \Sigma$. For $g\in G$, define the \emph{$T$--length of $g$} as  $\|g\|_T=\min\{k\co g = g_1 g_2 \cdots g_k \text{ where }g_i\in T\}$. By convention, we set $\|1\|_T=0$.  The $T$--distance between two elements $g$ and $h$ of $G$ is defined to be $d_T(g,h)=\|gh^{-1}\|_T$.  It is clear that $d_T$ is a right-invariant metric on $G$.

\begin{lem}\label{lem:elec}
\begin{enumerate}
\item
The metric space $(G, d_T)$ is quasi-isometric to $(\hat{Y}, \hat{d})$.
\item
If $\diam_X(A\cup A\Sigma)<\infty$, then $X$ is connected and $(G, d_T)$ is quasi-isometric to $(X, d_X)$.
\end{enumerate}
\end{lem}

\begin{rem}\label{rem:qi}
All the quasi-isometries in the following proof will be $G$-equivariant.
\end{rem}

\begin{proof}
(1) is obvious from the existence of a continuous surjection from $\cay(G,T)$ onto $\hat Y$ which restricts to the natural inclusion on the vertex set.

For (2), let us fix $\alpha_0\in A$ and define $\psi\co G\to X$ by $\psi(g) = \alpha_0g$.
It is obvious that the image is quasi-dense.
Put  \[M = \max(\{d_X(\alpha_0,\alpha_0s)\co s\in \Sigma\}\cup\{2 \diam_X(A)\}).\]  We have that $M\le 2\diam_X(A\cup A\Sigma)<\infty$.
For $\alpha\in A$ and $h\in H_\alpha$, we have
$d_X(\alpha_0,\alpha_0 h)\le d_X(\alpha_0,\alpha)+d_X(\alpha,\alpha_0 h) 
= d_X(\alpha_0,\alpha)+d_X(\alpha h,\alpha_0 h) \le 2\diam_X(A)\le M$.
So we see that $d_X(\alpha_0,\alpha_0 g)\le M\|g\|_T$ for each $g\in G$.
In particular for every $\alpha,\beta\in A$ and $x,y\in G$, we have
$d_X(\alpha x,\beta y) \le d_X(\alpha x,\alpha_0 x) + d_X(\alpha_0 x,\alpha_0 y) + d_X(\alpha_0 y,\beta y) 
\le 2\diam_X (A) + M \|y x^{-1}\|_T<\infty$. Hence $X$ is connected.

Choose a (finite) set $B$ of representatives of edge orbits for $G$--action on $X$
and write $B =\{\{\alpha_ix_i,\beta_i y_i\}\co i=1,2,\ldots\}$ where $\alpha_i,\beta_i\in A$ and $x_i,y_i\in G$.
Set \[M' = \max\{\|y_i x_i^{-1}\|_T\co i=1,2,\ldots\}.\]
Consider an edge in $X$ of the form $\{\alpha,\beta h\}$ where $\alpha,\beta\in A$ and $h\in G$.
There exists $\{\alpha x,\beta y\}\in B$ and $g\in G$ such that $\alpha = \alpha xg$ and $\beta h = \beta yg$.
Since $xg\in H_\alpha$ and $h(yg)^{-1}\in H_\beta$,
we have $\|h\|_T = \|h (yg)^{-1} (yx^{-1}) (xg)\|_T \le M'+2$.
This shows that for every $g\in G$ we have $\|g\|_T \le (M'+2) d_X(\alpha_0,\alpha_0 g)$.
\end{proof}

It is often useful to consider the case when $X$ is not connected.
Let $\alpha\ne \beta\in X^{(0)}$.
Suppose that $\ell$ is the smallest nonnegative integer such that
 there exist $g_1,g_2,\ldots,g_\ell\in G$ satisfying
(i) $\alpha\in Ag_1$ and $\beta\in Ag_\ell$,
and 
(ii) $A g_i\cap Ag_{i+1}\ne\varnothing$ for $i=1,2,\ldots,\ell -1$.
Then $\ell$ is called the \emph{covering distance between $\alpha$ and $\beta$} and denoted as $d'_X(\alpha,\beta)$ or $d'(\alpha,\beta)$.
We define $d'(\alpha,\alpha) = 0$.

\begin{lem}\label{lem:cov}
If $\cup_{\alpha\in A} H_\alpha$ generates $G$, then $(G,d_T)$ is quasi-isometric to $(X^{(0)},d')$.
\end{lem}

\begin{proof}
We may assume $\Sigma\subseteq \cup_{\alpha\in A} H_\alpha$.
Fix $\alpha_0\in A$.
We will prove  $\|g\|_T -1 \le  d'(\alpha_0,\alpha_0g) \le  \|g\|_T+1$ for each $g\in G$.
Note that if $h\in G$ and $\alpha\in A\cap A h$, then $h\in H_\alpha$.
Now consider $g\in G$ and set $\ell= d'(\alpha_0,\alpha_0 g)$.
We can choose $h_1, h_2,\ldots,h_\ell$ such that
$\alpha_0\in A h_1$, $\alpha_0 g\in A h_\ell \cdots h_2 h_1$
and $h_i\in H_{\alpha_i}$ where $\alpha_i\in A$ for $i=1,2,\ldots,\ell$.
We see that $ g(h_\ell \cdots h_2 h_1)^{-1}\in H_{\alpha_0}$.
Then $g = g(h_\ell \cdots h_2 h_1)^{-1}\cdot h_\ell\cdots h_2 h_1$ has $T$--length at most $\ell+1$.

Conversely, suppose we have written $g = h_k\cdots h_2 h_1$ where for each $i$ we have $\alpha_i\in A$ and $h_i\in H_{\alpha_i}$.
Then each consecutive pair in the sequence 
\[(\alpha_0,\alpha_1, \alpha_2 h_1,\ldots,\alpha_k h_{k-1} \cdots h_2 h_1,  \alpha_0 g)\]
 is contained in the image of $A$ by some element of $G$. This shows that $d'(\alpha_0, \alpha_0 g)\le \|g\|_T+1$.
\end{proof}

Let us define a graph $X'$  with  $V(X')=V(X)$
by the following condition:
 for $x,y\in G$ and $\alpha,\beta\in A$,
 the two vertices $\alpha x,\beta y$ are declared to be adjacent if there exists $g\in G$ such that
$\alpha x= \alpha g$ and $\beta y = \beta g$.

\begin{lem}\label{lem:cnt}
\begin{enumerate}
\item
For $\alpha,\beta\in V(X')=V(X)$, we have $d_{X'}(\alpha,\beta) = d'_X(\alpha,\beta)$.
\item
If $\cup_{\alpha\in A} H_\alpha$ generates $G$, then $X'$ is connected.
\end{enumerate}
\end{lem}

\begin{proof}
Obvious from the definition and Lemma~\ref{lem:cov}.
\end{proof}

So in the above sense, the graph $X'$ is a ``connected model'' for $X$.
This will become useful in Section~\ref{s:proj} when we metrize the extension graph of a disconnected graph.

\begin{cor}\label{cor:all qi}
If $X$ is connected and $\cup_{\alpha\in A} H_\alpha$ generates $G$, then 
the following metric spaces are all quasi-isometric to each other:
\[(G,d_T), (\hat Y,\hat d), (X,d_X),(X^{(0)},d'),(X',d_{X'}).\]
\end{cor}

\subsection{Star length}
The following is now immediate from Corollary~\ref{cor:all qi}.

\begin{lem}[{\cite[Lemma 3.2]{masurminsky1}}]\label{lem:mm qi}
Let $S$ be a surface. Fix a generating set $\Sigma$ for $\Mod(S)$ such that $\Sigma=\Sigma^{-1}$ and a set of representatives $A$ of vertex orbits in the curve graph $\mC (S)$. 
Let $\hat Y$ be the electrification of $\cay(\Mod(S),\Sigma)$ with respect to $\coprod_{\alpha\in A}\stab(\alpha)\setminus \Mod(S)$ and put $T = \cup_{\alpha\in A}\stab(\alpha)\cup \Sigma$.
Then the following metric spaces are quasi-isometric to each other:
\[(\Mod(S),d_T),(\mC (S),d_{\mC(S)}),(\hat Y,\hat d).\]
Since $\mC (S)$ is $\delta$--hyperbolic, it follows that $\Mod(S)$ is weakly--hyperbolic relative to $\{\stab(\alpha)\co \alpha\in A\}$.
\end{lem}

Let $\Gam$ be a finite connected graph.
Consider the (right--)conjugation action of $A(\Gamma)$ on $X = \Gamma^e$. Each vertex of $ \Gamma^e$ is in the orbit of a vertex in $\Gamma$ and for $v\in V(\Gamma)$ we have $\stab(v) = \form{\st{v}}$.
The \emph{star length} of $g\in A(\Gamma)$ is the minimum $\ell$ such that $g$ can be written as the product of $\ell$ elements in $\cup_{v\in V(\Gamma)} \form{\st{v}}$. 
We write the star length $g$ as $\stl{g}$.  The star length induces a metric $d_*$ on $\aga$.
From Corollary~\ref{cor:all qi}, we have the following.

\begin{thm}\label{thm:qi}
Let $\hat Y$ be the electrified Cayley graph of $A(\Gamma)$ with respect to the generating set $V(\Gamma)$ and the collection of cosets $\coprod_{v\in V(\Gamma)} \form{\st v}\setminus A(\Gamma)$.
Then the metric spaces
$(A(\Gamma),d_*), (\Gamma^e,d_{\Gam^e})$ and $(\hat{Y},\hat d)$ are quasi--isometric to each other.
\end{thm}

With regards to Theorem \ref{thm:qi}, we note that $\gex$ is connected whenever $\Gam$ is connected.  In \cite{KK2012}, the authors proved that $\Gam^e$ is a quasi-tree. Therefore, we immediately obtain the following fact:

\begin{cor}\label{cor:coned off qi tree}
The group $A(\Gamma)$ is weakly hyperbolic relative to $\{\form{\st{v}}\co  v\in  V(\Gamma)\}$.
\end{cor}

We remark that by the work of M. Hagen (\cite{hagenquasi}, Corollary 5.4), we have that $\aga$ is weakly hyperbolic relative to $\{\form{\lk{v}}\co  v\in  V(\Gamma)\}$.
Suppose $\Gamma$ has no isolated vertices. For $g\in A(\Gamma)$, we let $\|g\|_{\mathrm{link}}$ denote the word length of $g$ with respect to the generating set $\cup_{v\in V(\Gamma)}\form{\lk{v}}$.
Then $\stl{g}\le \|g\|_{\mathrm{link}}\le 2\stl{g}$ and so, Corollary \ref{cor:coned off qi tree} can be obtained from Hagen's result. If $\Gamma$ has an isolated vertex, then $\cup_{v\in V(\Gamma)}\form{\lk{v}}$ does not generate $A(\Gamma)$.
The general case of weak hyperbolicity of Artin groups (not necessarily right-angled ones) relative to subgroups generated by subsets of the vertices of the defining graph was studied by Charney and Crisp in \cite{charneycrisp}.

\section{Combinatorial group theory of right-angled Artin groups}\label{s:comb}
\subsection{Reduced words}
Let $\Gam$ be a finite graph. 
Each $g\in A(\Gam)$ can be represented by a \emph{word}, which is a sequence $s_1 s_2\cdots s_\ell $ where each $s_i\in V(\Gam)\cup V(\Gam)^{-1}$ for $i=1,2,.\ldots,l$. 
Each term $s_i$ of the sequence is called a \emph{letter} of the word.
The \emph{word length} of $g$ is the smallest length of a word representing $g$, and denoted as $|g|$.
A word $w$ is \emph{reduced} if the length of $w$ is the same as the word length of the element which $w$ represents.
A word is often identified with the element of $A(\Gam)$ which the word represents. 
An element of $A(\Gam)$ will be sometimes assumed to be given as a reduced word, if the meaning is clear from the context.
If $w_1, w_2,\ldots,w_k$ are words, we denote by $w_1 \cdot w_2 \cdots w_k$ the word obtained by concatenating $w_1, w_2, \ldots,w_k$ in this order.
If two words $w$ and $w'$ are equal as words, then we write $w\equiv w'$.

\begin{defn}\label{defn:red}
A sequence $(g_1, g_2,\ldots,g_\ell )$ of elements in $A(\Gam)$ is called \emph{reduced} if $|g_1 g_2\cdots g_\ell  | = |g_1|+|g_2|+\cdots+|g_\ell |$.
In this case, we write $x\sim (g_1, g_2,\ldots,g_\ell )$ where $x = g_1 g_2\cdots g_\ell \in A(\Gam)$.
\end{defn}

Note that $g\sim (\cdots,p,q,\cdots)$ implies that $g\sim (\cdots, pq,\cdots)$. Also, if $g\sim (g_1, g_2, \ldots, g_\ell )$ and each $g_i$ is represented by a reduced word, then $g_1 \cdot g_2\cdots g_\ell $ is a reduced word $g\in A(\Gam)$.

The \emph{support} of a word $w$ is the set of vertices $v$ such that $v$ or $v^{-1}$ is a letter of $w$.
The \emph{support} of a group element $g$ is the support of a reduced word for $g$.
We define the \emph{support} of a sequence $(g_1, g_2, \ldots, g_s)$ of elements in $A(\Gam)$ as the sequence
$(\supp(g_1),\supp(g_2),\ldots,\supp(g_s))$.

If a word $w'$ is a subsequence of another word $w'$, then we write $w'\preccurlyeq_0 w$.
In particular, if $w'$ is a consecutive subsequence of $w$, then we say $w'$ is a \emph{subword} of $w$ and write $w'\preccurlyeq w$.  Here, we say that $w'$ is a \emph{consecutive subsequence} of $w$ if $w=w_1\cdot w'\cdot w_2$ for possibly trivial words $w_1$ and $w_2$ (that is to say, the generators used to express $w'$ can be arranged to occur consecutively, without any interruptions).
We define analogous terminology for group elements:

\begin{defn}\label{defn:subword}
Let $g,h\in A(\Gam)$ and $A\subseteq V(\Gam)$.
\begin{enumerate}
\item
If $g\sim (x,h,y)$ for some $x,y\in A(\Gam)$, then we say $h$ is a \emph{subword} of $g$ and write $h\preccurlyeq g$.
\item
If a reduced word for $h$ is a subsequence of some reduced word for $g$, then we say $h$ is a \emph{subsequence} of $g$
and write $h\preccurlyeq_0 g$.
\item
If $g\sim (x,y)$, then we say $x$ is a \emph{prefix} of $g$ and $y$ is a \emph{suffix} of $g$.
We write $x\preccurlyeq_p g$ and $y\preccurlyeq_s g$.
\item
We write $\tau(g;A) = h$ if (i) $h\preccurlyeq_s g$, (ii) $\supp(h)\subseteq A$ and (iii) if $h'$ is another element of $A(\Gam)$ satisfying (i) and (ii) and $|h| \le |h'|$, then $h = h'$.
\item
We write $\iota(g;A) = h$ if (i) $h\preccurlyeq_p g$, (ii) $\supp(h)\subseteq A$ and (iii) if $h'$ is another element of $A(\Gam)$ satisfying (i) and (ii) and $|h| \le |h'|$, then $h = h'$.
\end{enumerate}
\end{defn}

\subsection{More on star length}\label{ss:morestar}
Suppose $\Gam$ is a finite connected graph.  We will write $d$ for the graph metric $d_{\Gam^e}$ on the extension graph.
We have defined the star length of a word in Section~\ref{s:elec}.
By a \emph{star--word}, we mean a word of star length one.
If $g$ and $g'$ are two elements of $A(\Gam)$, we say that $\Gam^{gg'}$ is the \emph{conjugate of $\Gam^g$ by $g'$}. 
By regarding $V(\Gamma)$ as a set of vertex--orbit representatives, we have the notion of 
the covering distance on $\Gamma^e$ as defined in Section~\ref{s:elec}.
More concretely, let $x$ and $y$ be vertices of $\Gam^e$.
Suppose $\ell$ is the smallest nonnegative integer such that there exist conjugates $\Gam_1,\Gam_2,\ldots, \Gam_\ell $ of $\Gam$ in $\Gam^e$ satisfying
(i)
$x\in \Gam_1$ and $y\in\Gam_\ell $
and (ii)
$\Gam_i\cap\Gam_{i+1}\ne\varnothing$ for $i=1,2,\ldots,\ell -1$.
Then $\ell$ is the \emph{covering distance between $x$ and $y$} and denoted as $d'(x,y)$.
As we have seen in Lemma~\ref{lem:cov}, we have a quasi-isometry between $(\Gamma^e,d')$ and $(A(\Gamma),d_*)$. We can see more precise estimates as follows,
which are immediate consequences of the proof of Lemma~\ref{lem:cov}.
\begin{lem}\label{lem:covering distance}
Let $v\in V(\Gam)$, $x,y\in V(\Gam^e)$ and $g\in A(\Gam)$.
\begin{enumerate}
\item
$\|g\|_* -1 \le d'(v,v^g) \le \|g\|_* +1$.
\item
$d'(x,y)\le d(x,y)\le \diam(\Gamma) d'(x,y)$.
\item
$\|g\|_* -1 \le d(v,v^g) \le \diam(\Gamma)(\|g\|_*+1)$.
\end{enumerate}
\end{lem}

\begin{rem}
Note that the connectivity assumption for $\Gamma$ is not used for the covering distance
and Lemma~\ref{lem:covering distance} (1). This observation will become critical later.
\end{rem}

\begin{lem}\label{lem:subseq}
Let $x,y,z\in A(\Gam)$. 
\begin{enumerate}
\item
If $x\preccurlyeq_0 y$, then $\|x\|_*\le \|y\|_*$.
\item
For $\ell = \|x\|_*$, there exist star--words $x_1, x_2, \ldots, x_\ell$ such that $x\sim (x_1,x_2,\ldots,x_\ell)$.
\item
Suppose $xy\sim(x,y)$ and $\|y\|_*\ge \|z\|_*+2$.
Then $xyz \sim (x,yz)$.
\end{enumerate}
\end{lem}
\begin{proof}
(1) If $y = y_1 \cdots y_k$ for some star--words $y_i$, then $x = x_1\cdots x_k$ for some $x_i\preccurlyeq_0 y_i$.
(2) is obvious from (1). For (3), assume $(x,yz)$ is not reduced. 
Write $y\sim(y_1,y_2),z\sim(y_2^{-1},z_1)$ such that $yz\sim (y_1,z_1)$.
For some $p\in   V(\Gamma)\cup V(\Gamma)^{-1}$, we can write
 $x\sim (x_0,p),z_1\sim(p^{-1},z_2)$ and $[p,y_1]=1$.
Then $\stl{y}\le\stl{y_1}+\stl{y_2}\le1+\stl{z}$.
\end{proof}

We briefly explain the notion of a \emph{dual van Kampen diagram} $\Delta$ for a word $w$ representing a trivial element in $A(\Gam)$; see~\cite{CW2004, Kim2008, Kim2010} for more details. The diagram $\Delta$ is obtained from a van Kampen diagram $X\subseteq S^2$ for $w$ by taking the dual complex of $X$ and removing a small disk around the vertex $\infty$ corresponding to the exterior region of $X\subseteq S^2$; see Figure~\ref{fig:vK} for an example.

\begin{figure}[htb]
  \tikzstyle {av}=[red,draw,shape=circle,fill=red,inner sep=0pt]
  \tikzstyle {bv}=[blue,draw,shape=circle,fill=blue,inner sep=0pt]
  \tikzstyle {cv}=[orange,draw,shape=circle,fill=orange,inner sep=0pt]
  \tikzstyle {wv}=[black,draw,shape=circle,fill=white,inner sep=1pt]  
  \tikzstyle {gv}=[black,draw,shape=circle,fill=black,inner sep=1pt]  
  \tikzstyle {v}=[black,draw,shape=circle,fill=black,inner sep=2pt]  
  \tikzstyle {a}=[red,postaction=decorate,decoration={%
    markings,%
    mark=at position 1 with {\arrow[red]{stealth};}}]
  \tikzstyle {b}=[blue,postaction=decorate,decoration={%
    markings,%
    mark=at position .85 with {\arrow[blue]{stealth};},%
    mark=at position 1 with {\arrow[blue]{stealth};}}]
  \tikzstyle {c}=[orange,postaction=decorate,decoration={%
    markings,%
    mark=at position .7 with {\arrow[orange]{stealth};},%
    mark=at position .85 with {\arrow[orange]{stealth};},
    mark=at position 1 with {\arrow[orange]{stealth};}}]
  \tikzstyle {am}=[red,postaction=decorate,decoration={%
    markings,%
    mark=at position .5 with {\arrow[red]{stealth};}}]
  \tikzstyle {bm}=[blue,postaction=decorate,decoration={%
    markings,%
    mark=at position .45 with {\arrow[blue]{stealth};},%
    mark=at position .55 with {\arrow[blue]{stealth};}}] 
  \tikzstyle {pv}=[black,draw,shape=rectangle,fill=black,inner sep=2pt] 
  \tikzstyle {ar}=[postaction=decorate,decoration={%
    markings,%
    mark=at position .35 with {\arrow[black,ultra thick]{stealth};}}]
  \tikzstyle {cmid}=[orange,postaction=decorate,decoration={%
    markings,%
    mark=at position .45 with {\arrow[orange]{stealth};},%
    mark=at position .5 with {\arrow[orange]{stealth};},
    mark=at position .55 with {\arrow[orange]{stealth};}}] 
\subfloat[(a) A van Kampen diagram $X$]{
	\begin{tikzpicture}[scale=.75,thick]
	\foreach \x in {0,2} {\draw [am] (-2,\x)--(0,\x); }	
	\foreach \x in {-2,0,2} {\draw [bm] (\x,0)--(\x,2); }	
	\foreach \x in {0,2} {\draw [cmid] (0,\x)--(2,\x); }	
	\draw [red,dashed] (0,4) edge [out=-120,in=90] (-.5,1);
	\draw [red,dashed] (-.5,1) edge [out=-90,in=0] (-2,-1);
	\draw [red,dashed] (-2,-1) edge [out=180,in=-90] (-4,3);
	\draw [red,dashed] (-4,3) edge [out=90,in=150] (0,4);	
	\draw [orange,dashed] (0,4) edge [out=-60,in=90] (.5,1);
	\draw [orange,dashed] (.5,1) edge [out=-90,in=180] (2,-1);
	\draw [orange,dashed] (2,-1) edge [out=0,in=-90] (4,3);
	\draw [orange,dashed] (4,3) edge [out=90,in=30] (0,4);

	\draw [blue,dashed] (0,4) edge [out=-45,in=90] (3,2);
	\draw [blue,dashed] (3,2) edge [out=-90,in=0] (0,1.5);
	\draw [blue,dashed] (0,1.5) edge [out=180,in=-90] (-3,2);
	\draw [blue,dashed] (-3,2) edge [out=90,in=-135] (0,4);

	\draw (-2,0) node [left=3,below] {\small $x_0$};
	\draw (0,4) node [gv] {} node [right] {\small $\infty$};	
	\draw [ultra thick,dashed] (0,4) circle (1);
	\foreach \x in {-2,0,2} {\draw (\x,2) node [gv] {}; }	
	\foreach \x in {-2,0,2} {\draw (\x,0) node [gv] {}; }	
	\node  [inner sep=0.9pt] at (-2.5,0) {}; 	
	\node  [inner sep=0.9pt] at (2.5,0) {}; 		
	\end{tikzpicture}}
$\qquad$
\subfloat[(b) A dual van Kampen diagram $\Delta$]{
	\begin{tikzpicture}[scale=.75,thick]
	\draw [ar] (1.732,-1) arc [radius=2, start angle = -30, end angle=30];
	\draw [ar] (1.732,1) arc [radius=2, start angle = 30, end angle=90];
	\draw [ar] (0,2) arc [radius=2, start angle = 90, end angle=150];
	\draw [ar] (-1.732,-1) arc [radius=2, start angle = 210, end angle=150];
	\draw [ar] (0,-2) arc [radius=2, start angle = 270, end angle=210];
	\draw [ar] (1.732,-1) arc [radius=2, start angle = 330, end angle=270];
	
	\node [av] at (0:2) (r) {};
	\node [cv] at (60:2) (ru) {};
	\node [bv] at (120:2) (lu) {};
	\node [cv] at (180:2) (l) {};
	\node [av] at (240:2) (ld) {};
	\node [bv] at (300:2) (rd) {};
	\foreach \x in {30,90,150,210,270,330} {\draw (\x:2) node [v] {};}

	\foreach \x/\xtext in {
            0/a,60/c,120/b,180/c^{-1},240/a^{-1},300/b^{-1}}
                \draw (\x:2.5) node [] {$\xtext$};
	
	\draw [red] (r) edge [bend right] (ld);
	\draw [orange] (ru) edge [bend left] (l);
	\draw [blue] (lu) edge [bend left] (rd);	
 	\node  [inner sep=0.9pt] at (-2.5,-3) {}; 	
	\node  [inner sep=0.9pt] at (2.5,-3) {}; 			
	\end{tikzpicture}}
\caption{Dualizing a van Kampen diagram for the word $w=acbc^{-1}a^{-1}b^{-1}$.}
\label{fig:vK}
\end{figure}
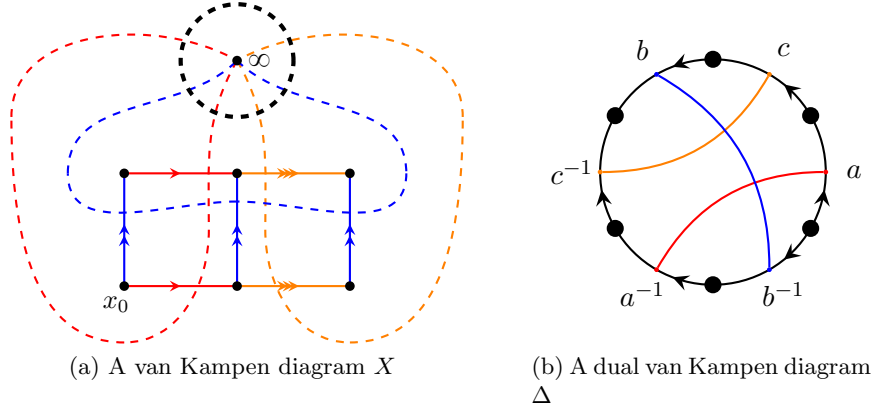

Since $X$ is a simply connected square complex, $\Delta$ is homeomorphic to a disk equipped with properly embedded arcs (\emph{hyperplanes}). 
So the boundary of $\Delta$ is topologically a circle,
which is oriented and divided into segments. Each segment is labeled by a letter of $w$ so that if one follows the orientation of $\partial \Delta$ and reads the labels, one gets $w$ with a suitable choice of the basepoint. 
Each arc $\alpha$ joins two segments on $\partial \Delta$ whose labels are inverses to each other.
Such an arc $\alpha$ is labeled by the vertex that is the support of the two letters. 
If two arcs intersect, then their vertices are adjacent in $\Gam$.
We often identify a letter with the corresponding segment if the meaning is clear.

The \emph{syllable length} of a word $g$ is the minimum $\ell$ such that $g$ can be written as the product of $\ell$ elements in $\cup_{v\in V(\Gamma)} \form{v}$.
We write the syllable length of $g$ as $\syl{g}$.
The syllable length of a word can be estimated from a decomposition into star--words in the following sense.

\begin{lem}\label{lem:syl}
For $g\in A(\Gam)$ and $r = \stl{g}$, we have
\[
\syl{g} = \inf \{ \sum_{i=1}^r \syl{h_i} \co g = h_1 h_2\cdots h_r\text{ and }\|h_i\|_*=1\}.
\]
\end{lem}

\begin{proof}
Let $\ell=\syl{g}, r = \stl{g}$ and $\ell'$ be the right-hand side of the equation.
It is obvious that $\ell\le \ell'$.
Write $g = x_1^{e_1}x_2^{e_2}\cdots x_\ell ^{e_\ell }$ where $x_i$ are vertices of $\Gam$ and $e_i\ne0$.
Consider star--words $h_1,\ldots,h_r$ such that 
$g\sim(h_1,\ldots,h_r)$.
There exists a dual van Kampen diagram $\Delta$ of the word
$x_1^{e_1}\cdot x_2^{e_2}\cdots x_\ell ^{e_\ell }\cdot h_r^{-1}\cdots h_1^{-1}$.
Let $A_i = \{ j \co \text{ there is an arc from }x_i^{e_i}\text{ to }h_j^{-1}\text{ in }\Delta\}$.
We may choose $(h_1,h_2,\ldots,h_r)$ and $\Delta$ such that $\sum_{j=1}^l |A_i|$ is minimal.

If a pair of neighboring arcs starting at $x_i^{e_i}$ are joined to a letter in $h_j$ and to a letter in $h_{j'}$ for some $j<j'$, 
then the interval between those two letters contains letters in $\lk(x_i)\cup \lk(x_i)^{-1}$.
In this case, we replace $h_j$ by $h_j x_i^{\sgn(e_i)}$ and $h_{j'}$ by $x_i^{-\sgn(e_i)} h_{j'}$. Note that $\sum_i |A_i|$ did not increase.
By repeating this ``left-greedy'' process and using the minimality, we see that $|A_i|=1$ for $i=1,2,\ldots,l$. This shows that for each $j$, there exists $k>0$ and $i_1<i_2<\cdots<i_k$ such that 
$h_j \sim (x_{i_1}^{e_{i_1}},x_{i_2}^{e_{i_2}},\ldots,x_{i_k}^{e_{i_k}})$. This gives a partition of $\{x_1^{e_1}, x_2^{e_2}, \ldots, x_\ell ^{e_\ell }\}$ into $r$ disjoint sets, and hence,
$\sum_{j=1}^r \syl{h_j} \le l$.
\end{proof}

Let us record a lemma that will be used later in this paper when we establish a version of the distance formula for certain classes of right-angled Artin groups (see Proposition \ref{p:disttree}, for instance). The proof is a simple exercise.

\begin{lem}\label{lem:star}
Let $\Gamma$ be a finite, triangle-- and square--free graph.
Suppose $y\in V(\Gamma)$
and $g\in \form{\st(y)}\subseteq A(\Gamma)$.
\begin{enumerate}
\item
If $|\supp(g)\cap\lk(y)|\ge2$, then $\Gamma\cap\Gamma^g=\{y\}$.
\item
If $\supp(g)\cap\lk(y)=\{a\}$ and $\syl{y}>1$, then $\Gamma\cap\Gamma^g=\{a,y\}$. 
\item
If $\supp(g)\cap\lk(y)=\varnothing$ and $g\ne1$, then $\Gamma\cap\Gamma^g=\st(y)$. 
\item
If $h\in A(\Gamma)$ is not a star--word, then $\Gamma\cap\Gamma^g=\varnothing$.
\end{enumerate}
\end{lem}

\subsection{Quasi--isometries of right-angled Artin groups with the star length metric}
In this subsection, we give a quasi--isometry classification of right-angled Artin groups equipped with the star length metric.  It turns out that this classification is much coarser than the quasi--isometry classification of right-angled Artin groups with the word metric~\cite{BJN2010,BKS2008}. In particular, there are only two quasi--isometry classes. We will denote by $T_\infty$ the simplicial tree which has countable valance at each vertex.

\begin{thm}\label{thm:starqiclass}
Let $\Gamma$ be a connected finite simplicial graph and let $\aga$ be the associated right-angled Artin group, equipped with the star length metric.  Then $\aga$ is quasi--isometric to exactly one of the following:
\begin{enumerate}
\item
A single point.
\item
The countable regular tree $T_{\infty}$.
\end{enumerate}
\end{thm}
\begin{proof}
By Theorem \ref{thm:qi}, the group $\aga$ equipped with the star length metric is quasi--isometric to $\gex$.  Recall that $\gex$ has finite diameter if and only if $\aga$ splits as a nontrivial direct product or if $\aga\cong\Z$.

Suppose $\diam(\gex)=\infty$.
We have a quasi-isometry $\phi\co \gex\to T$ for some tree $T$.
Every vertex of $\gex$ has a quasi--dense orbit under the $\aga$ action.  
Furthermore, it is easy to check that for $v\in V(\gex)$, the graph $\gex\setminus \st(v)$ has infinitely many components $C_1, C_2, \ldots$ of infinite diameter.

If we remove a ball $B_K(v)$ of radius $K$ about a vertex $v\in \gex$, then the minimal distance between $C_i\setminus B_K(v)$ and $C_j\setminus B_K(v)$ grows like $K$ for $i\ne$j.  
So, if $K$ is chosen to be much larger than the quasi--isometry constants for a quasi--isometry $\phi:\gex\to T$, we see that $\gex\setminus B_K(v)$ has infinitely many vertices which must be sent to pairwise distinct vertices of a bounded subset of $T$ under $\phi$.  
If $T$ is locally finite, any bounded subset of $T$ is finite, a contradiction.  Thus for each vertex $w$ of $T$, there is a uniform $M$, which is independent of $w$, such that the ball of radius $M$ about $w$ has infinitely many vertices. In particular, such an $M$--ball contains at least one vertex of infinite degree. For each vertex of finite valence, choose a path to a nearest vertex of infinite valence. Such a path has length at most $M$. Hence, contraction of such paths will yield a quasi-isometry from $T$ to $T_\infty$.
\end{proof}

In Subsection \ref{ss:disc}, we will define an appropriate notion of distance on disconnected extension graphs, with respect to which even totally disconnected extension graphs will be quasi--isometric to $T_{\infty}$.

\section{Acylindricity of the $\aga$ action on $\Gamma^e$}
Let $\Gam$ be a finite graph.
The goal of this section is to show that the conjugation action of $A(\Gam)$ on $\Gam^e$ is acylindrical. 

\begin{defn}\label{defn:acyl}
An isometric action of a group $G$ on a path-metric space $X$ is called \emph{acylindrical} if for every $r>0$, there exist $R ,N>0$ such that whenever $p$ and $q$ are two elements of $X$ with $d(p,q)\ge R$, the cardinality of the set 
$\xi(p,q;r) = \{g\in G\co d(p,g.p)\le r\mbox{ and }d(q,g.q)\le r\}$ is are most $N$.
\end{defn}

All the real variables ($r,s,t,R,N,M$ and so forth) will be assumed to take integer values throughout this section.
We say a vertex set $A$ is \emph{adjacent} to another vertex set $B$ if each vertex of $A$ is adjacent to every vertex of $B$; in particular, $A$ and $B$ are disjoint sets because of the assumption that all graphs under consideration are simplicial.

\begin{defn}\label{defn:cancel}
Let $\beta=(s,x,y,v_1,v_2,\ldots,v_s)$ where $s>0,x,y\in A(\Gam)$ and $v_1,v_2,\ldots v_s\in V(\Gam)$.
We say the sequence 
$\alpha=(g_1,g_2,\ldots,g_s,h_1,h_2,\ldots,h_s)\in A(\Gamma)^{2s}$
 is a \emph{cancellation sequence} for $g\in A(\Gam)$ with respect to $\beta$ if the following four conditions hold.
\begin{enumerate}[(i)]
\item
$g = g_1 g_2 \cdots g_s h_1 h_2 \cdots h_s$.
\item
$\supp(g_i), \supp(h_i)\subseteq \st(v_i)$ for each $i$.
\item
For each $1\le i<j\le s$, we have that $\supp(h_i)$ is adjacent to $\supp(g_j)$.
\item
$x\sim (x',w,g_s^{-1},g_{s-1}^{-1},\ldots,g_1^{-1})$ and 
$y\sim (h_s^{-1},h_{s-1}^{-1},\ldots,h_1^{-1},w^{-1},y')$ and for some $x',y',w\in A(\Gam)$.
\end{enumerate}
Furthermore, we say $\alpha$ is \emph{maximal} if the lexicographical order of 
\[(|g_1|,|g_2|,\ldots,|g_s|,|h_s|,|h_{s-1}|,\ldots,|h_1|)\]
 is maximal among the cancellation sequences for $g$ w.r.t. $\beta$.
\end{defn}

If there is a cancellation sequence for $g$ w.r.t.~$\beta$, then it is obvious that a maximal one exists since the complexity is bounded above by
$(|x|,|x|,\ldots,|x|,|y|,|y|,\ldots,|y|)$. Roughly speaking, the item $(2)$ below means that $(g_1,\ldots,g_s)$ is ``left--greedy'' and $(h_1,h_2,\ldots,h_s)$ is ``right-greedy''. The proofs are straightforward.

\begin{lem}\label{lem:max}
Suppose $\alpha=(g_1,g_2,\ldots,g_s,h_1,h_2,\ldots,h_s)$ is a maximal cancellation sequence for some $g$ w.r.t. some $\beta$.
Let $1\le a<b\le s$.
\begin{enumerate}
\item
If $(g_1',g_2',\ldots,g_s', h_1',h_2',\ldots,h_s')$ is another cancellation sequence for $g$ w.r.t. $\beta$ and $|g_i|\le  |g_i'|, |h_i| \le |h_i'|$ for each $i$, then $g_i = g_i'$ and $h_i=h_i'$.
\item
Suppose that for some $u\in A(\Gam)$, either
\begin{enumerate}[(i)]
\item
$(g_1,\ldots,g_{a-1},g_a',g_{a+1},\ldots,g_{b-1},g_b',g_{b+1},\ldots,g_s,h_1,h_2,\ldots,h_s)$
is another cancellation sequence for $g$ w.r.t. $\beta$
where $g_{a}' \sim (g_{a}, u)$ and $g_{b} \sim (u^{-1}, g_{b}')$,
or
\item
$(g_1,\ldots,g_s,h_1,\ldots,h_{a-1},h_a',h_{a+1},\ldots,h_{b-1},h_b',h_{b+1},\ldots,h_s)$
is another cancellation sequence for $g$ w.r.t. $\beta$
where 
$h_a \sim (h_a', u)$ and $h_b' \sim (u^{-1}, h_b)$.
\end{enumerate}
Then $u = 1$.
\end{enumerate}
\end{lem}

For $t>0$, we define $B_t = \{g\in A(\Gam) : \|g\|_*\le t\}$ and $B'_t = \{g\in A(\Gam) : \|g\|_*< t\}$. These ``balls'' are infinite sets in general.

\begin{lem}\label{l:cancel}
Let $s, t  > 0$ and $M\ge s+t+2$.
Suppose $x,y\in A(\Gam)$ satisfy that $\|x\|_*\ge M$ or $\|y\|_*\ge M$.
If $g\in B_s\cap x^{-1}B_t y^{-1}$, then there exist vertices $v_1, v_2, \ldots, v_s$ of $\Gam$ and a cancellation sequence for $g$ with respect to $(s, x,y, v_1, v_2, \ldots, v_s)$. Moreover, we have $\|w\|_*,\|x\|_*, \|y\|_*\ge M-s-t$ where $w$ is as in Definition~\ref{defn:cancel}.
\end{lem}

Roughly speaking, the above lemma implies that if $x$ or $y$ is  long and $g$ and $xgy$ are both short in the star lengths,
then $x$ and $y$ are both long and $g$ must be ``completely cancelled'' in $xgy$. 

\begin{proof}[Proof of Lemma \ref{l:cancel}]
Without loss of generality, let us assume $\|x\|_*\ge M$.
There exist $v_1, v_2,\ldots, v_s$ in $V(\Gam)$ and $w_i\in \form{\st(v_i)}$ such that $g \sim (w_1, w_2, \cdots, w_s)$.
We can write $xgy \sim (z_1, z_2, \ldots,z_t)$ for some star--words $z_1, z_2, \ldots, z_t$.

Let $\Delta$ be a dual van Kampen diagram for the following word
\[x\cdot w_1\cdot w_2 \cdots w_s \cdot y \cdot z_t^{-1} \cdot z_{t-1}^{-1} \cdots z_1^{-1},\]
where $x,y, z_i$ are represented by reduced words.
Note that no two letters from \[w_1\cdots w_2\cdots w_s\] are joined by an arc in $\Delta$.
We call the interval on $\partial \Delta$ reading $x\cdot w_1\cdot w_2\cdots w_s\cdot y$ as $\partial_1$ and the closure of $\partial\Delta\setminus\partial_1$ as $\partial_2$. 

We have $x\sim (x', w,x'')$ such that the letters of $x', w$ and $x''$ are joined to $\partial_2, y$ and $g$, respectively.
We may assume that $x$ is written as the reduced word $x'\cdot w\cdot x''$.
Then $x'\preccurlyeq z_1 z_2\cdots  z_t$ and $x''\preccurlyeq w_s^{-1} w_{s-1}^{-1} \cdots w_1^{-1}$.
In particular, $\|w\|_* \ge \|x\|_* - s - t \ge M - s - t$. Since $w\preccurlyeq y$, we have $\|y\|_*\ge M - s -t$. This shows the ``moreover'' part of the lemma.

Write each $w_i\sim (g_i,u_i,h_i)$ where $g_i, u_i$ and $h_i$ are represented by reduced words such that the letters of $g_i, u_i$ and $h_i$ are joined to $x$, $\partial_2$ and $y$, respectively.
Suppose $v\in \supp(u_i)$. Then $\supp(w)$ is adjacent to $v$ and hence, $\|w\|_*\le 1$; this contradicts $M\ge s+t+2$.
Hence, $u_i=1$ for each $i=1,2,\ldots,s$. See Figure~\ref{fig:u}.

\begin{figure}[htb!]
\begin{center}
\includegraphics[width=.25\textwidth]{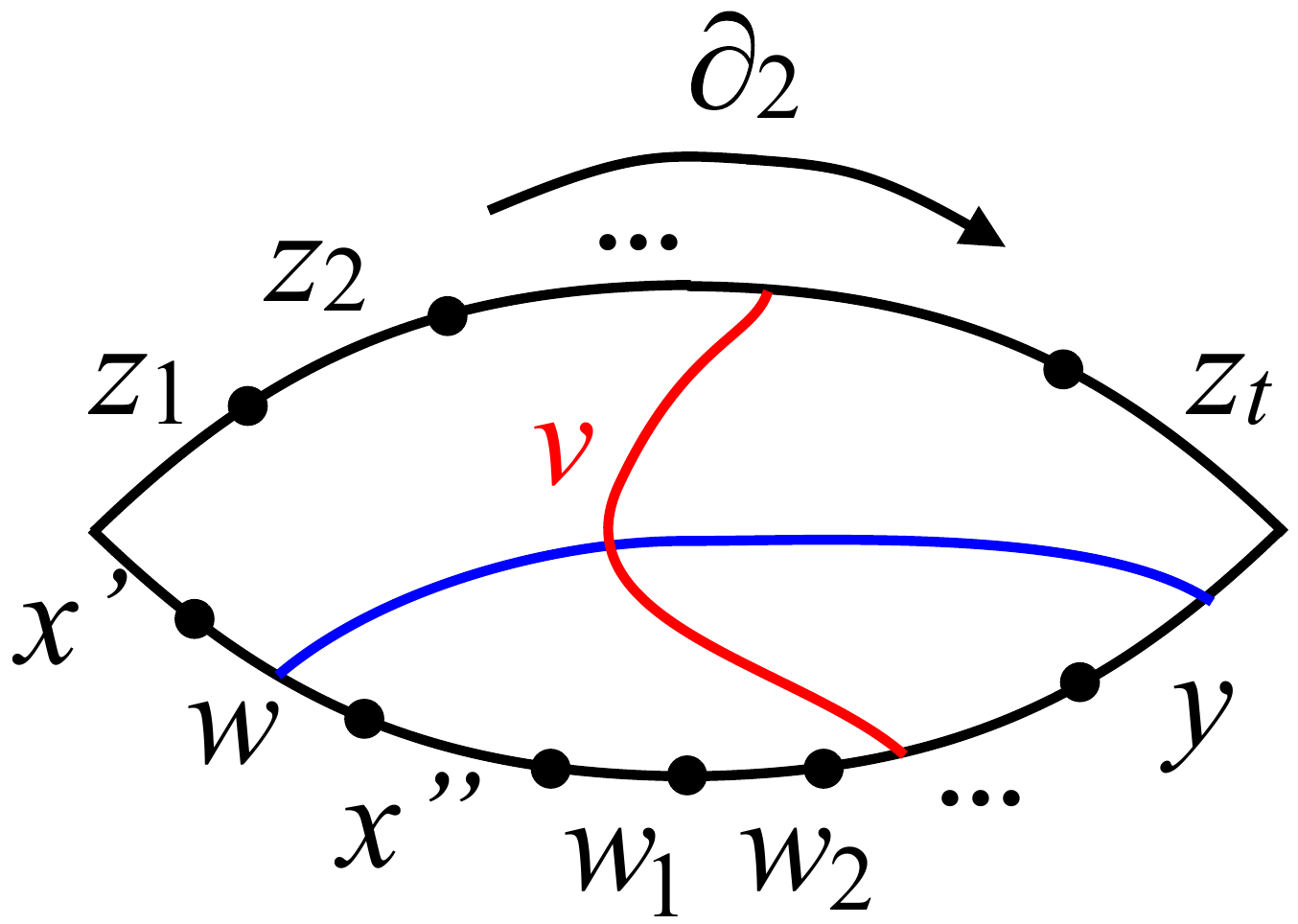}
  \caption{Proof of Lemma~\ref{l:cancel}.}
  \label{fig:u}
  \end{center}
\end{figure}

Let $1\le i<j\le s$. 
Each letter of $g_j$ is joined to a letter in $x$ by an arc separating $h_i$ from $y$; hence we have (iii).
Now (i), (ii) and (iv) are obvious from the construction and (iii).
\end{proof}

\begin{rem}\label{rem:cancel}
\begin{enumerate}
\item
In the above proof, we have $w_i\sim(g_i,h_i)$. However, we do not assume for \emph{maximal} $\alpha$ that each $g_i\cdot h_i$ is a reduced concatenation. This will be critical in the proof of Lemma~\ref{lem:pq}.
\item
Write $x\sim (p,p')$ where $\|p'\|_*=s+2$, a hypothesis stronger than that of Lemma \ref{l:cancel}.
Then no letter of $p$ is joined to a letter of $g$ by Lemma~\ref{lem:subseq} (3). Hence $g_s^{-1} g_{s-1}^{-1}\cdots g_1^{-1}\preccurlyeq_s p'$. Similarly if we write $y \sim (q',q)$ so that $\|q'\|_* = s+2$, then $h_s^{-1} h_{s-1}^{-1}\cdots h_1^{-1}\preccurlyeq_p q'$.
Hence, the cardinality of $B_s\cap x^{-1}B_t y^{-1}$ is at most the number of the choices for $g_1, g_2, \ldots, g_s\preccurlyeq (p')^{-1}$ and $h_1, h_2, \ldots, h_s\preccurlyeq (q')^{-1}$, and so bounded above by $(2^{|p'|+|q'|})^s$.
The goal of Lemma~\ref{lem:st} is to find another upper-bound depending only on $\Gam,s$ and $t$.
\end{enumerate}
\end{rem}

We define the \emph{support} of a sequence of words as the sequence of the supports of those words.

\begin{lem}\label{lem:pq}
Let $\beta= (s,x,y,v_1, v_2,\ldots, v_s)$ as in Definition~\ref{defn:cancel} and $t>0$.
Suppose $\stl{x}\ge s+t+2$ or $\stl{y}\ge s+t+2$.
Let $P_i, Q_i\subseteq\st(v_i)$ for each $i=1,2,\ldots,s$
such that $Q_i$ is adjacent to $P_j$ for each $1\le i<j\le s$.
Then there exists at most one element $g$ in $x^{-1}B_ty^{-1}$ such that $(P_1, P_2, \ldots, P_s, Q_1, Q_s, \ldots, Q_s)$ is the support of a maximal cancellation sequence for $g$ with respect to $\beta$.
\end{lem}

\begin{proof}
Suppose $g$ is such an element.
Define $x_0 = x$ and inductively, $p_i = \tau(x_{i-1};P_i), x_i = x_{i-1}p_i^{-1}$ for $i=1,2,\ldots,s$.
Similarly, put $y_0=y$ and $q_i  = \iota(y_{i-1};Q_i), y_i = q_i^{-1}y_{i-1}$ for $i=1,2,\ldots,s$.
We will prove the conclusion by showing that $g = p_1^{-1}p_2^{-1}\cdots p_s^{-1}q_1^{-1}q_2^{-1}\cdots q_s^{-1}$. Choose a cancellation sequence $\alpha=(g_1, g_2,\ldots, g_s, h_1, h_2, \ldots,h_s)$ for $g$ with respect to $(s; x,y; v_1, v_2, \ldots, v_s)$ with support 
$(P_1, P_2, \ldots, P_s, Q_1, Q_2, \ldots, Q_s)$.
Let $z$ be a reduced word representing $xgy$
and $\Delta$ be a dual van Kampen diagram for the following word:
\[x\cdot g_1\cdot h_1 \cdots g_s\cdot h_s \cdot y \cdot z^{-1}.\]
We use an induction on $i=1,2,\ldots, s$ to prove that there is a one-to-one correspondence between the letters of $g_i$ and those of $p_i$ such that each corresponding pair are joined by an arc in $\Delta$, and hence, $g_i=p_i^{-1}$.
By the inductive hypothesis and the maximality of $p_i$, 
 each letter of $g_i$ is joined to a letter of $p_i$.
Suppose a letter $u$ of $p_i$ is not joined to any of the letters of $g_i$. 
We may assume $u\in V(\Gamma)$, rather than $u\in V(\Gamma)^{-1}$.
Let $w$ be as in Definition~\ref{defn:cancel}.
If $u$ is joined to a letter in $z$ then every letter of $w$ would be adjacent to $u$. Since $\stl{w}\ge2$ by Lemma~\ref{l:cancel}, this is impossible. Therefore, we only have the following two cases.

\emph{Case 1.} $u$ is joined to a letter $u^{-1}$ of $g_j$ for some $i<j\le s$ by an arc, say $\gamma$; see Figure~\ref{fig:pq} (a).

We may assume every letter appearing after $u$ in $p_i$ is joined to a letter in $g_i$; namely, $u\cdot g_i^{-1}\preccurlyeq_s p_i$.
Choose an arbitrary letter $u'$ between the last letter of $g_i\preccurlyeq g$ and the letter $u^{-1}$ of $g_j$. The arc $\gamma'$ originating from $u'$ ends at a letter of $x$ before $u$ or at a letter of $y$. Then $\gamma'$ intersects with 
$\gamma$ and $[u',u]=1$. We can enlarge $g_i$ by setting $g_i' \sim (g_i, u^{-1})$ and $g_j' = ug_j$.
Note that $\supp(g_i') = P_i$ and $\supp(g_j')\subseteq P_j$ and hence, the condition (iii) of Definition~\ref{defn:cancel} is satisfied.
This contradicts to the left-greediness of $(g_1, g_2, \ldots, g_s)$.

\emph{Case 2.} $u$ is joined to a letter of $y$; see Figure~\ref{fig:pq} (b).

For each $j>i$, the arc $\gamma$ intersects with every arc originating from $g_j$ and hence, the vertex $u$ is adjacent to $P_j$.
Since $u\subseteq P_i$, the vertex $u$ is adjacent to $Q_1, Q_2, \ldots, Q_{i-1}$.
We set $g_i'\sim (g_i,u^{-1}), h_i' \sim (u, h_i)$; see Figure~\ref{fig:pq} (c). 
Note that $\supp(g_i') = P_i = \supp(g_i)$ and $\supp(h_i')\subseteq \{u\}\cup Q_i$.
The condition (iii) of Definition~\ref{defn:cancel} is obvious.
This again contradicts to the maximality of the sequence $\alpha$.

We conclude that $g_i = p_i^{-1}$ for $i=1,2,\ldots, s$. We also see that $h_i = q_i^{-1}$ by symmetry.

\begin{figure}[htb!]
\begin{center}
\subfloat[(a) Case 1.]{\includegraphics[width=.8\textwidth]{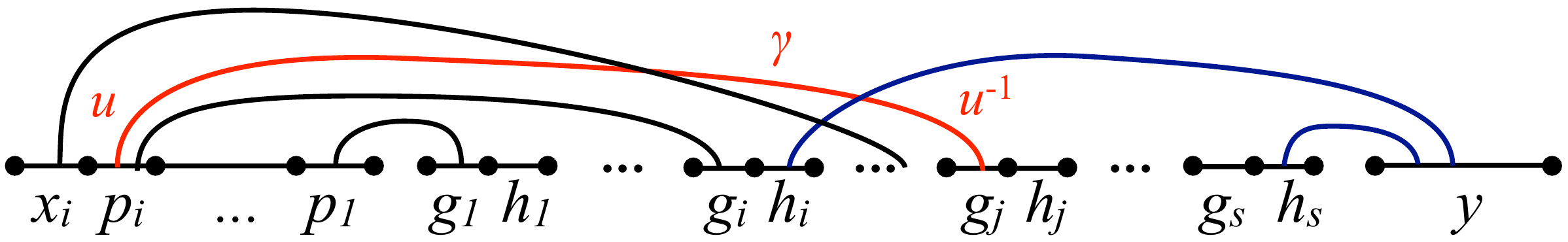}}\\
\subfloat[(b) Case 2.]{\includegraphics[width=.8\textwidth]{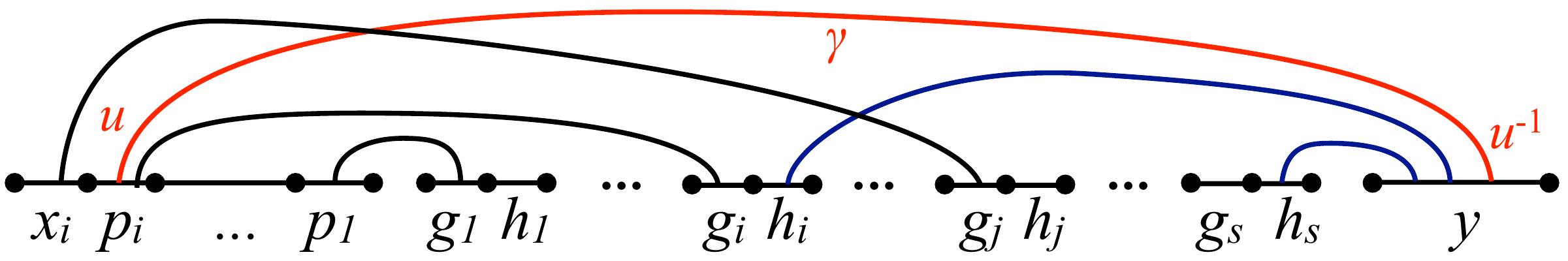}}\\
\subfloat[(c) Case 2 after change.]{\includegraphics[width=.8\textwidth]{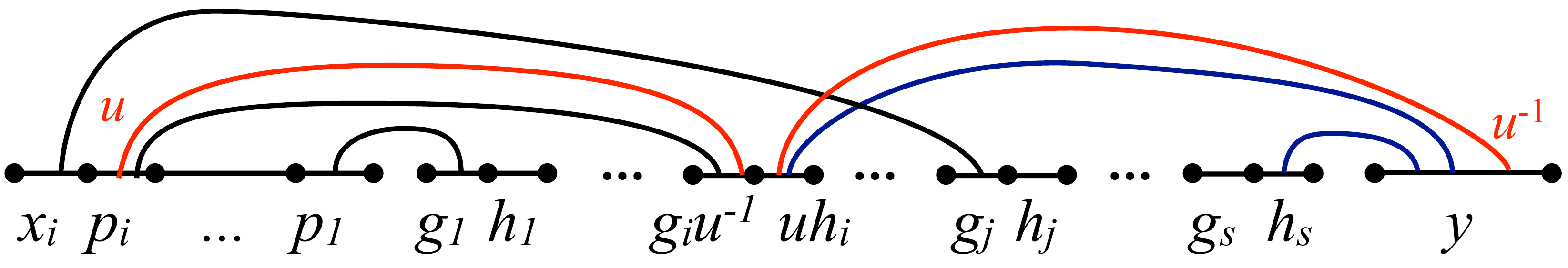}}
  \caption{Proof of Lemma~\ref{lem:pq}.}
  \label{fig:pq}
  \end{center}
\end{figure}

\end{proof}

\begin{lem}\label{lem:st}
Let $s,t>0$ and $x,y\in A(\Gam)$ such that $\|x\|_*\ge s+t+2$ or $\|y\|_*\ge s+t+2$.
Then the cardinality of $B_s\cap x^{-1}B_t y^{-1}$ is at most 
$|V(\Gam)|^s (2^{|V(\Gam)|})^{2s}$.
\end{lem}

\begin{proof}
In Lemma~\ref{l:cancel},
the number of the possible choices for $v_1, v_2, \ldots, v_s$ is bounded by $|V(\Gam)|^s$.
Also, the number of ways of choosing $P_1, P_2, \ldots, P_s, Q_1, Q_2, \ldots, Q_s$ in Lemma~\ref{lem:pq}
is at most $(2^{|V(\Gam)|})^{2s}$.
\end{proof}

\begin{thm}\label{thm:acyl}
The action of $A(\Gam)$ on $\Gam^e$ is acylindrical.
\end{thm}

\begin{proof}
Let us fix a vertex $v$ of $\Gam$ and let $r>0$ be given. 
Put $D=\diam(\Gamma),s = r+2D+1, R = D(2s+5)$ and $N =|V(\Gam)|^s (2^{|V(\Gam)|})^{2s}$.
Suppose $p$ and $q$ are two vertices of $\Gam^e$ such that $d(p,q)\ge R$.
Then there exist $w',w\in A(\Gam)$ such that $d(p,v^{w'})\le D$ and $d(q,v^w)\le D$.
Without loss of generality, we may assume $w'=1$.
By Lemma~\ref{lem:covering distance}, we have $\|w\|_* \ge d(v,v^w)/D - 1\ge(R-2D)/D - 1\ge 2s+2$.
For every $g\in \xi(p,q;r)$, 
we have $\|g\|_*\le d(v,v^g)+1 \le  d(v,p) + d(p,p^g) + d(p^g, v^g) + 1\le r + 2D + 1=s$
and similarly, $\|wgw^{-1}\|_* \le d(v,v^{wgw^{-1}}) +1 = d(v^w, v^{wg})+1
 \le  d(v^w,q) + d(q,q^g) + d(q^g, v^{wg}) + 1
 \le s$.
Hence, $\xi(p,q;r)\subseteq  w^{-1} B_s w\cap B_s$.
Since $\|w\|_*\ge 2s+2$, Lemma~\ref{lem:st} implies that the set
 $\xi(p,q;r)$ has at most $N$ elements.
\end{proof}

\section{Nielsen--Thurston classification}\label{s:ntclass}
We have shown that the action of the right-angled Artin group on the extension graph $\Gam^e$ is acylindrical.  We will now discuss several consequences of the acylindricity of the action.

Recall that if a group $G$ acts on a metric space $(X,d)$ by isometries, we may consider the \emph{translation length} (also called \emph{stable length}) of an element $g\in G$.  This quantity is defined for an arbitrary choice of $x\in X$ as following.
\[\tau(g)=\lim_{n\to\infty}\frac{1}{n}d(x,g^n(x)).\]  
This limit always exists and independent of $x$. The following proposition appears as Lemma 2.2 of \cite{bowditchtight}.
\begin{prop}
Let $G$ be a group acting acylindrically on a $\delta$--hyperbolic graph $X$.  Then for each nonidentity $g\in G$, either $g$ is \emph{elliptic}, in which case $\tau(g)=0$ and one (and hence every) orbit of $g$ is bounded, or $g$ is \emph{loxodromic}, in which case $\tau(g)\geq\epsilon>0$.   The constant $\epsilon$ depends only on the acylindricity parameters of the action and the hyperbolicity constant of $X$.
\end{prop}

The principal content of the previous proposition is that acylindrical actions of groups on hyperbolic graphs do not have parabolic elements.  The previous proposition gives us a way to formulate a Nielsen--Thurston classification for nonidentity elements of a right-angled Artin group.  The following tables summarize the Nielsen--Thurston classification for mapping classes.

\begin{table}
\begin{center}
\begin{tabular}{ | p{2cm}| p{2cm}| p{2cm}| p{5cm}|}
\hline
\multicolumn{4}{|c|}{{\bf Nontrivial mapping class $\psi$}} \\
\hline
{\bf $\mC(S)$ type} & {\bf Nielsen--Thurston classification} & {\bf Curve complex characterization} & {\bf Intrinsic algebraic characterization (in $\Out(\pi_1(S))$)} \\ \hline
Elliptic & Finite order & Every orbit is finite & Some nonzero power of $\psi$ is the identity\\ \hline
Elliptic & Infinite order reducible & There exists a finite orbit and an infinite orbit & Some nonzero power of $\psi$ preserves the conjugacy class of a nonperipheral, nontrivial isotopy class in $\pi_1(S)$\\ \hline
Loxodromic & Pseudo-Anosov & Stable length is nonzero & No power of $\psi$ preserves any nonperipheral, nontrivial isotopy class in $\pi_1(S)$\\
\hline
\end{tabular}
\caption{Nielsen--Thurston classification in MCGs}
\end{center}
\end{table}

The following table summarizes the analogous Nielsen--Thurston classification for nonidentity elements of right-angled Artin groups:
\begin{table}
\begin{center}
\begin{tabular}{ | p{2cm}| p{3cm}| p{6.5cm}|}
\hline
\multicolumn{3}{|c|}{{\bf Cyclically reduced, nonidentity element $g$ in $\aga$}} \\
\hline
{\bf $\gex$ type} &{\bf  Extension graph characterization} & {\bf Intrinsic algebraic characterization (in $\aga$)} \\ \hline
Elliptic &  Every orbit is bounded & Support of $g$ is contained in a nontrivial subjoin of $\Gamma$\\ \hline
Loxodromic & Stable length is nonzero & Support of $g$ is not contained in a subjoin of $\Gamma$\\
\hline
\end{tabular}
\caption{Nielsen--Thurston classification in RAAGs}
\end{center}
\end{table}

Now let us give a proof of the intrinsic algebraic characterization of loxodromic and elliptic elements of $\aga$.

\begin{lem}\label{l:anti--loop}
Let $(a_1,a_2,\ldots,a_\ell ,a_1)$ be an edge--loop in $\Gam\opp$ such that $\|a_1a_2\cdots a_\ell \|_*>1$.
Then for arbitrary $e_1, e_2\ldots,e_\ell \ne0$, we have that  $\|(a_1^{e_1}a_2^{e_2}\cdots a_\ell ^{e_\ell })^n\|_* > n$.
\end{lem}

\begin{proof}
Write $g=(a_1^{e_1}a_2^{e_2}\cdots a_\ell ^{e_\ell })^n = w_1  w_2 \cdots w_k$ where  $k=\|g\|_*$ and $\|w_i\|_*=1$ for each $i$.
Note that $g$ has only one reduced word representation.
Since $\{a_1,\ldots,a_\ell\}$ is not contained in one star, $\syl{w_i}\le \ell-1$ for each $i$.
Hence, 
$n\ell =\syl{g} \le \sum_i \syl{w_i}\le (\ell-1)k$.
We have $k\ge n\ell / (\ell-1) > n$.
\end{proof}

\begin{lem}\label{l:anti--loop2}
Let $g\in \aga$ be a cyclically reduced element such that the support of $g\in A(\Gam)$ is not contained in a join.
Then for each $n>0$, we have $\|g^{2n |V(\Gam)|^2}\|_*\ge n$.
\end{lem}

\begin{proof}
Put $M=2 (|V(\Gam)|-1)^2$.
Find an edge--loop $(a_1,a_2,\ldots,a_\ell ,a_1)$ in $\Gam\opp$ such that $\supp(g) = \{a_1,a_2,\ldots,a_\ell \}$.
We can further require that $\ell\le M$, since there exists an edge--loop in $\Gamma$ of length at most $M$ which visits all the vertices in a connected component of $\Gam\opp$.
Note that $a_i$'s may be redundant.
We will regard $g^M$ as the reduced word obtained by concatenating $M$ copies of a reduced word for $g$.
We can choose $M$ occurrences of $a_1^{\pm1}$ in $g^M$, all corresponding to the same letter of $g$.
Since $[a_1,a_2]\ne1$, we can choose $M-1$ occurrences of $a_2^{\pm1}$, alternating with the previously chosen $M$ occurrences of $a_1^{\pm1}$.
We can then similarly choose $M-2$ occurrences of $a_3^{\pm1}$ alternating with the previously chosen $M-1$ occurrences of $a_2^{\pm1}$, and continue.
Since $M-\ell+1\ge1$, it is easy to deduce that $a_1^{\pm1}a_2^{\pm2}\cdots a_\ell ^{\pm1}\preccurlyeq_0 g^M$.
By Lemma~\ref{lem:subseq} (1) and~\ref{l:anti--loop}, we have $\|g^{Mn}\|_*\ge n$.
Note that $2|V(\Gamma)|^2>M$.
\end{proof}

\begin{lem}\label{l:join}
Let $g\in \aga$ be a cyclically reduced element.  Then $\{\|g^n\|_*\}_{n\in\mathbb{N}}$ is bounded if and only if the support of $g$ is contained in a join.
\end{lem}
\begin{proof}
Assume that the support of $g$ is contained in a join $J=J_1*J_2\le \Gamma$. 
Let us choose $v_i\in V(J_i)$ so that $J_i\subseteq \st(v_{3-i})$ for $i=1,2$.
We can write $g = g_1 g_2$ so that $g_i\in A(J_i)\le\form{\st(v_{3-i})}$ for each $i$.
Then $\|g^n\|_* \le \|g_1^n\|_*+\|g_2^n\|_*\le 2$ for every $n$.

The opposite direction is obvious from Lemma~\ref{l:anti--loop2}.
\end{proof}

\begin{thm}\label{t:ntclass}
Let $1\neq g\in \aga$ be cyclically reduced.  Then $g$ is elliptic if and only if the support of $g$ is contained in a join of $\Gamma$.
\end{thm}
\begin{proof}
An element $1\neq g\in \aga$ is elliptic if and only if for every vertex $v$ of $\gex$, the orbit $\{v^{g^n}\}_{n\in\Z}$ is bounded.  For any $g\in \aga$, the distance $d_{\gex}(v,v^g)$ is coarsely equal to the star length $\stl{g}$.  It follows that if each orbit $\{v^{g^n}\}_{n\in\Z}$ is bounded then the star lengths $\stl{g^n}$ are uniformly bounded.  By Lemma \ref{l:join}, we see that this occurs if and only if $\supp(g)$ is contained in a join in $\Gamma$.
\end{proof}

It is interesting to note a fundamental difference between the action of $\Mod(S)$ on $\mC(S)$ and the action of $\aga$ on $\gex$.  Observe that if $\psi\in\Mod(S)$ is nontrivial and elliptic, then $\psi$ has a finite orbit in $\mC(S)$.  If $g\in\aga$ is elliptic, it may not have a finite orbit, even though each orbit is bounded in $\gex$.

\begin{prop}\label{p:nofixedpoint}
Consider \[\aga=F_2\times F_2=\langle a,b\rangle\times\langle c,d\rangle.\]  The element $abcd\in F_2\times F_2$ is elliptic and has no finite orbits in $\gex$.
\end{prop}

\begin{proof}
Every nontrivial element of $\aga$ is elliptic, since $\Gamma$ splits as a join.  Actually, $\gex$ has finite diameter in this case.  
By adding a degree-one vertex $v$ to $a$, 
one obtains a graph $\Lambda$ which does not split as a  join and thus satisfies $\diam(\Lambda^e)=\infty$.  
It is clear from the Centralizer Theorem that if $g\in A(\Lambda)$ and $n\neq 0$ then $(abcd)^n$ does not stabilize $v^g$.  Thus, it suffices to show that for each nonzero $n$, the element $(abcd)^n$ does not stabilize any conjugate of any of the vertices $\{a,b,c,d\}$.

Let $w\in\{a,b,c,d\}$, and let $g\in A(\Lambda)$.  The centralizer of the vertex $w^g\in\Lambda^e$ is generated by $\st(w^g)\subseteq\Lambda^e$.  However, the element $abcd$ and all of its nonzero powers are not supported in the star of any vertex of $\Lambda^e$.
\end{proof}

We now briefly sketch another approach to the Nielsen--Thurston classification for elements of right-angled Artin groups which relies on the work of M. Hagen and which was pointed out by the referee.  In \cite{hagenquasi}, Hagen developed the \emph{hyperplane graph} for a general CAT(0) cube complex $X$.  This graph has a vertex for every hyperplane $H\subset X$, and two hyperplanes are adjacent if they \emph{contact}, which is to say if they are dual to a pair of $1$--cubes which share a common $0$--cube. 
Let us now assume $\Gamma$ is a finite simplicial graph without isolated vertices.
We denote the hyperplane graph of the universal cover of the Salvetti complex of $\aga$
by $\Gamma^{hyp}$.
At the end of Section 4, we noted that $\|\cdot\|_{\mathrm{link}}$ and $\|\cdot\|_*$ define quasi-isometric distances on $\aga$, say $d_\mathrm{link}$ and $d_*$.
Since $(\aga,d_\mathrm{link})$ and $(\aga,d_*)$ are quasi-isometric to  $\Gamma^\mathrm{hyp}$ and $\gex$ respectively,
we see that $\Gamma^{\mathrm{hyp}}$ and $\gex$ are quasi-isometric.

This quasi-isometry can be realized by a graph map $f:\gex\to \Gamma^{\mathrm{hyp}}$
which sends a vertex conjugate $v^g$ to the hyperplane dual to the generator $v$, translated by the action of $g$.  
In \cite{hagenbdy}, Hagen proves that any element $g$ of a group $G$ acting on a CAT(0) cube complex $X$ either has a bounded orbit in  $\Gamma^{\mathrm{hyp}}$ or has a quasi--geodesic axis (Theorem 5.2).  In the former case, $g$ stabilizes a hyperplane or a join.  
It follows that an element $1\neq g\in\aga$ either has an unbounded orbit in $\Gamma^{\mathrm{hyp}}$ and hence an unbounded orbit in $\gex$, or it stabilizes a hyperplane or a join and thus has a bounded orbit in $\gex$.  Thus, Hagen's hyperplane graph can be used to establish the Nielsen--Thurston classification in $\aga$.
We remark however, that the hyperplane graph is not quasi--isometric to the extension graph if $\Gam$ has an isolated vertex. For example, when $\Gam$ consists of a single vertex $\Gamma^{\mathrm{hyp}}$ is a real line while $\gex=\Gam$.

\section{Injective homomorphisms and types}\label{s:type}
\subsection{Type preservation under injective homomorphisms}
In this subsection we consider the behavior of elliptic and loxodromic elements under elements of $\Mono(G,H)$, where $G$ and $H$ are allowed to be right-angled Artin groups and mapping class groups, and where $\Mono$ denotes the set of injective homomorphisms from $G$ to $H$.  Namely, we consider $f\in\Mono(G,H)$ and $1\neq g\in G$.  Since $g$ is either elliptic or loxodromic, we consider whether $f(g)$ is also elliptic or loxodromic.

Table~\ref{table:type} summarizes type preservation under injective homomorphisms.
The case where $\Mod(S)$ is the source and $\aga$ is the target is handled by the following result (see \cite{Koberda2012} and the references therein):

\begin{prop}\label{p:noinj}
Let $S$ be a surface with genus $g\geq 3$ or $g=2$ with at least two punctures, let $G<\Mod(S)$ be a finite index subgroup, and let $\aga$ be a right-angled Artin group.  Then there is no injective homomorphism $G\to\aga$.
\end{prop}

\begin{table}
\begin{center}
\begin{tabular}{cc|p{4.3cm}|p{4.3cm}|l}
\cline{3-4}
& & \multicolumn{2}{c|}{{\bf Target}} \\ \cline{3-4}
& & $\aga$ & $\Mod(S)$ & \\ \cline{1-4}
\multicolumn{1}{|c}{\multirow{2}{*}{{\bf Source}}} &
\multicolumn{1}{|c|}{$\aga$} & Elliptics preserved (Theorem~\ref{t:ellipticpreserved}), loxodromics not preserved (Proposition~\ref{p:loxtype}) & Elliptics preserved (Theorem~\ref{t:ellipticpreserved}), loxodromics not preserved (Proposition~\ref{p:p4s2})  &    \\ \cline{2-4}
\multicolumn{1}{|c}{}                        &
\multicolumn{1}{|c|}{$\Mod(S)$} & Usually no injective homomorphisms (Proposition~\ref{p:noinj}) & Elliptics preserved (Theorem~\ref{t:ellipticpreserved}), loxodromics not preserved (\cite{aramayonaleiningersouto}, Theorem 2) &    \\ \cline{1-4}
\end{tabular}
\caption{Type (non)--preservation under injective homomorphisms}\label{table:type}
\end{center}
\end{table}

\begin{thm}\label{t:ellipticpreserved}
Let $f\in\Mono(G,H)$, where $G$ and $H$ are right-angled Artin groups or mapping class groups, and let $1\neq g\in G$ be elliptic.  Then $f(g)$ is elliptic.
\end{thm}
\begin{proof}
Observe that without loss of generality, we may replace $g$ with any positive power.  In particular, whenever $G$ is a mapping class group, we may assume that $g$ is a pure mapping class.  Reducible pure mapping classes and elliptic elements in $\aga$ are characterized by the fact that their centralizers are not (virtually) cyclic.  Since $f$ is injective and $g$ is elliptic, the centralizer of $f(g)$ in $H$ will not be virtually cyclic.  It follows that $f(g)$ is also elliptic.
\end{proof}

For the non--preservation of loxodromics, we have the following examples, the first of which is due to Aramayona, Leininger and Souto:

\begin{thm}[\cite{aramayonaleiningersouto}, Theorem 2]
There exists an injective map between two mapping class groups under which a pseudo-Anosov mapping class is sent to a Dehn multitwist.
\end{thm}

For the next proposition, let $P_4$ denote the path on four vertices and let $S_2$ denote a surface of genus two.

\begin{prop}\label{p:p4s2}
There exists an injective map $A(P_4)\to \Mod(S_2)$ whose image contains a pseudo-Anosov mapping class and under which a loxodromic element is mapped to a reducible element.
\end{prop}
\begin{proof}
Label the vertices of $P_4$ in a row by $\{a,b,c,d\}$ as in Figure~\ref{fig:s2} (a).  We consider a chain of four simple closed curves on $S_2$ which fills $S_2$.  The curves should then be labelled in order by $\{B,D,A,C\}$; see Figure~\ref{fig:s2} (b).

\begin{figure}[htb!]
\begin{center}
\subfloat[(a)]{\includegraphics[width=.2\textwidth]{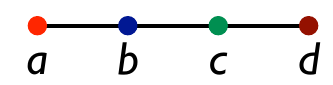}}
\hspace{.5in}
\subfloat[(b)]{\includegraphics[width=.35\textwidth]{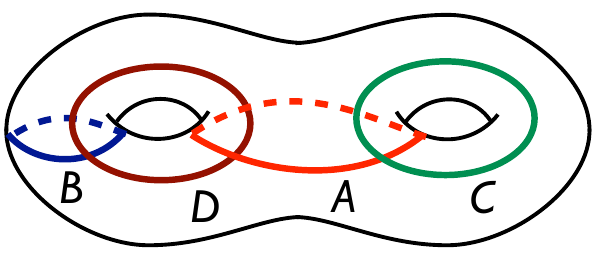}}
\caption{Proposition~\ref{p:p4s2} and Corollary~\ref{cor:p4s2}.}
  \label{fig:s2}
  \end{center}
\end{figure}

 The injective map $\iota: A(P_4)\to\Mod(S_2)$ is given by taking the vertex generator labelled by a lower case letter $x$ to a sufficiently high power of a Dehn twist about the curve labelled by the upper case letter $X$.  Thus, for some sufficiently large $N$, we have an injective map $\iota$ defined by \[a\mapsto T_A^N,\, b\mapsto T_B^N,\, c\mapsto T_C^N,\, d\mapsto T_D^N.\]
Since $\{A,B,C,D\}$ fill $S_2$, there is a pseudo-Anosov mapping class in the image of $A(P_4)$.  Since $A$ and $D$ intersect exactly once, they cannot fill $S_2$ (in fact they fill a torus with one boundary component).  The element $ad\in A(P_4)$ is loxodromic since the distance between $a$ and $d$ in $P_4$ is three, and the image of $ad$ is supported on a torus with one boundary component.  It follows that the image of $ad$ is reducible and hence elliptic in $\Mod(S_2)$.
\end{proof}

In \cite{KK2012}, the authors showed that for any graph $\Gamma$, there is a surface $S$ such that $\gex$ embeds in $\mC (S)$.  The previous proposition has the following corollary, which shows that embeddings from $\gex$ into curve graphs may not be well--behaved from a geometric standpoint:

\begin{cor}\label{cor:p4s2}
There exists an embedding $P_4^e\to \mC (S_2)$ which is not metrically proper.
\end{cor}
\begin{proof}
Consider the embedding of $A(P_4)$ into $\Mod(S_2)$ as given in Proposition \ref{p:p4s2}.  We obtain a map from $P_4^e$ to $\mC (S)$ as follows: send the vertices $\{a,b,c,d\}$ to the curves $\{A,B,C,D\}$ respectively.  We have that sufficiently high powers of Dehn twists about those curves generate a copy of $A(P_4)<\Mod(S_2)$.  We can obtain the double of $P_4$ over the star of any vertex by applying the corresponding power of a Dehn twist to the curves $\{A,B,C,D\}$, and adding the new curves to the collection $\{A,B,C,D\}$.  The map from the double of $P_4$ to $\mC (S)$ is given by sending a vertex to the corresponding curve.

We identify the vertex generators $\{a,b,c,d\}$ of $A(P_4)$ with sufficiently high powers of Dehn twists about the curves $\{A,B,C,D\}$.  We then consider the orbits of these curves inside of $\mC(S_2)$ under the action of $A(P_4)$.  By definition, this gives us an embedding of $P_4^e$ into $\mC(S_2)$.  Specifically, let \[\phi:A(P_4)\to\Mod(S_2)\] be the embedding above which sends a vertex generator $v$ to the Dehn twist power $T^M_{\gamma(v)}$.  The map $P_4^e\to\mC(S_2)$ is given by \[v^g\mapsto \phi(g^{-1})(\gamma(v)).\]

Write $w=ad\in A(P_4)$.  We have seen that this element is loxodromic.  It follows that for any vertex $v\in P_4^e$, we have that the distances $d(v,v^{w^n})$ tend to infinity.  However, the curves $A$ and $D$ do not fill $S_2$, so that any mapping class written as a product of twists about $A$ and $D$ will be reducible, and each of its orbits in $\mC (S_2)$ will be bounded.  It follows that for any vertex $v\in P_4^e$, the distances between the images of $v$ and $v^{w^n}$ in $\mC(S_2)$ will remain bounded as $n$ tends to infinity.
\end{proof}

We briefly note the following fact which shows that injective maps between right-angled Artin groups do not preserve loxodromic elements:

\begin{prop}\label{p:loxtype}
Let $\Gam$ be a finite simplicial graph which does not split as a nontrivial join, so that $\aga$ contains at least one loxodromic element.  Then the canonical inclusion $\aga\to\aga\times\bZ$ is an injective map which does not preserve loxodromic elements.
\end{prop}
\begin{proof}
This is clear, because the defining graph of $\aga\times\bZ$ splits as a nontrivial join, so that $\aga\times\bZ$ contains no loxodromic elements.
\end{proof}

\subsection{Subgroups generated by pure elements}
Whereas it is not true that loxodromic elements are sent to pseudo-Anosov mapping classes under injective maps between right-angled Artin groups and mapping class groups, one can always find faithful representations of right-angled Artin groups into mapping class groups which map loxodromics to pseudo-Anosov mapping classes on the subsurface filled by the support of the loxodromic.

In fact, more is true.  
An element $1\neq g\in\aga$ is \emph{pure} if, after cyclic reduction, it cannot be written as a product of two commuting subwords. 
Put $X = \{g_1,\ldots,g_k\}\subseteq G$ and denote by $\Comm(X)$ the commutation graph of $X$.
We evidently get a surjective map \[A(\Comm(X))\to\langle g_1,\ldots,g_k\rangle\le G,\] given by sending the vertex of $\Comm(X)$ labelled by $g_i$ to the element $g_i$ (see \cite{KK2012} for more discussion of both of these definitions).  In this subsection we establish the following result:

\begin{prop}\label{p:pure}
Given a graph $\Gamma$, there exists a surface $S$ and a faithful homomorphism \[\phi:\aga\to\Mod(S)\] such that if $1\neq g\in\aga$ is pure then $\phi(g)$ is pseudo-Anosov on a subsurface of $S$.
\end{prop}

Here, the list $X$ is called \emph{irredundant} if no two elements of $X$ share a nonzero common power (or equivalently generate a cyclic subgroup of $\aga$).
Proposition \ref{p:pure} will then imply the following result, which is the analogue of the primary result of \cite{Koberda2012} for right-angled Artin groups:

\begin{thm}\label{thm:purepower}
Let $X=\{g_1,\ldots,g_k\}$ be an irredundant collection of nonidentity, pure elements of $\aga$.  Then for all sufficiently large $N$, we obtain an isomorphism \[A(\Comm(X))\cong\langle g_1^N,\ldots,g_k^N\rangle\le \aga.\]
\end{thm}

The following result is due to Clay, Leininger and Mangahas in \cite{CLM2012} (compare with \cite{Koberda2012}):

\begin{thm}[\cite{CLM2012}, Theorem 1.1]\label{thm:clm}
Let $X=\{f_1,\ldots,f_k\}\subseteq\Mod(S)$ be an irredundant collection of pseudo-Anosov mapping classes on connected subsurfaces of $S$.  Suppose furthermore that for $i\neq j$, there are no inclusion relations between $\supp(f_i)$ and $\supp(f_j)$, and that no component of $\partial\supp(f_i)$ coincides with any component of $\partial\supp(f_j)$ whenever $\supp(f_i)\cap\supp(f_j)=\emptyset$. 
Put $\Gamma=\Comm(X)$ and let $v_i\in V(\Gamma)$ be the vertex labeled by $f_i$.
Then for all sufficiently large $N$, 
the map
\[\phi:A(\Comm(X))\to\langle f_1^N,\ldots,f_k^N\rangle\le\Mod(S)\]
defined by $v_i\mapsto f_i^N$ is an isomorphism.
Furthermore, if $1\neq w\in\aga$ is cyclically reduced,
then $\phi(w)$ is a pseudo-Anosov mapping class on each component of the subsurface 
$\Fill(X_0)\subseteq S$
where $X_0=\{\supp(f_i)\co 1\le i\le k\text{ and }v_i\in \supp(w)\}$.
\end{thm}

Here, $\Fill(X_0)$ denotes the smallest subsurface (up to isotopy) filled by \[\bigcup_{S_i\in X_0} S_i.\]  The subsurface $\Fill(X_0)$ is given, up to isotopy, by taking the set--theoretic union of the subsurfaces in $X_0$ and filling in any nullhomotopic boundary curves.

\begin{proof}[Proof of Proposition \ref{p:pure}]
Taking a sufficiently large surface $S$, we can find a collection of elements $X$ of $\Mod(S)$ for which $\Comm(X)\cong\Gamma$ and which satisfy the hypotheses of Theorem \ref{thm:clm} (this fact is proved in \cite[Corollary 1.2]{CLM2012}.  The reader may also consult \cite{KK2012} for another discussion).  Let $1\neq w\in\aga$ be pure.  After cyclically reducing $w$, we have that $\supp(w)\subseteq\Gamma$ is not a join.  Identifying $\supp(w)$ with $X_0\subseteq X$, the purity of $w$ translates to the connectedness of $\Fill(X_0)$.  It follows that $\phi(w)$ is pseudo-Anosov on a connected subsurface of $S$.
\end{proof}

\begin{proof}[Proof of Theorem \ref{thm:purepower}]
Let $X=\{g_1,\ldots,g_k\}\subseteq \aga$ be an irredundant collection of nonidentity pure elements of $\aga$, and let $\phi:\aga\to\Mod(S)$ be an injective homomorphism coming from Theorem \ref{thm:clm}.  We have that $\phi(g_i)$ is pseudo-Anosov mapping class on a connected subsurface of $S$ for each $g_i\in X$.  One can either check that these partial pseudo-Anosov mapping classes again satisfy the conditions of Theorem \ref{thm:clm}, or one can appeal to the primary result of \cite{Koberda2012} to see that for all sufficiently large $N$, we obtain an isomorphism \[A(\Comm(X))\cong\langle g_1^N,\ldots,g_k^N\rangle\le \aga,\] as claimed.
\end{proof}

\section{North--south dynamics on $\partial\gex$}
\label{s:loxodromic}
The following is a standard result about pseudo-Anosov mapping classes:
\begin{prop}\label{p:pseudoanosovfree}
Let $\{\psi_1,\ldots,\psi_k\}$ be pseudo-Anosov elements of $\Mod(S)$ with the property that for $i\neq j$, we have that $\psi_i$ and $\psi_j$ share no common power.  Then for all sufficiently large $N$, we have \[\langle \psi_1^N,\ldots,\psi_k^N\rangle\cong F_k.\]
\end{prop}
\begin{proof}[Sketch of proof]
One approach is as follows: we have that for $i\neq j$, the mapping classes $\psi_i$ and $\psi_j$ stabilize pairs of measured laminations $\mL_i^{\pm}$ and $\mL_j^{\pm}$.  Furthermore, neither of $\mL_i^{\pm}$ coincides with either of $\mL_j^{\pm}$.  The action of a pseudo-Anosov $\psi_i$ on $\mathbb{P}\mathcal{ML}(S)$ is by north--south dynamics, with a source at $\mL_i^-$ and a sink at $\mL_i^+$.

The boundary of $\mC(S)$ is identified with the space of ending laminations $\mathcal{EL}(S)\subseteq\mathbb{P}\mathcal{ML}(S)$, and the north--south dynamics restricts to $\mathcal{EL}(S)$.  A straightforward ping--pong argument shows that sufficiently high powers of the mapping classes in question generate the expected free group.
\end{proof}

One does not need to consider the boundary of $\mC(S)$ to prove the previous proposition.  The north--south dynamics on $\mathbb{P}\mathcal{ML}(S)$ suffices.  The reason we mention the boundary of $\mC(S)$ is because we will be using the boundary of $\gex$ to prove the following analogue:

\begin{thm}\label{t:loxfree}
Let $\lambda_1,\ldots,\lambda_k$ be loxodromic elements of $\aga$ with the property that for $i\neq j$, we have that $\lambda_i$ and $\lambda_j$ share no common power.  Then for all sufficiently large $N$, we have \[\langle \lambda_1^N,\ldots,\lambda_k^N\rangle\cong F_k.\]
\end{thm}

The proof of Theorem \ref{t:loxfree} is identical to that of Proposition \ref{p:pseudoanosovfree}.  The only point which needs to be checked is that a loxodromic $\lambda\in\aga$ truly acts on $\partial\gex$ by north--south dynamics.  Observe that Theorem \ref{t:loxfree} follows immediately from Theorem \ref{thm:purepower}.  It is not the result we wish to analogize, but rather the method of proof:

\begin{lem}\label{l:northsouth}
Let $\lambda\in\aga$ be a loxodromic element.  Then there exists a unique pair of points $p^{\pm}\subseteq\partial\gex$ such that for any compact subset $K\subseteq \partial\gex\setminus\{p^{-}\}$ and for any open subset $U\subseteq\partial\gex$ containing $p^+$, there exists an $N$ such that $\lambda^N(K)\subseteq U$.  Furthermore, if $\lambda'\in\aga$ is another loxodromic which shares no common power with $\lambda$, then no fixed boundary point of $\lambda$ coincides with any fixed boundary point of $\lambda'$.
\end{lem}

We will first gather the necessary facts and then prove Lemma \ref{l:northsouth}.  Theorem \ref{t:loxfree} will follow immediately.

\begin{lem}\label{l:cycliccentralizer}
Let $\lambda\in\aga$ be loxodromic.  Then the centralizer of $\lambda$ is cyclic.
\end{lem}
\begin{proof}
This follows from the Centralizer Theorem (see \cite{Servatius1989} or \cite[Lemma 5.1]{BC2010}) since after cyclic reduction, the support of $\lambda$ is not contained in a subjoin of $\Gamma$.
\end{proof}

It follows that two loxodromics commute if and only if they share a common power. 
Following~\cite{BH1999}, let us write $a=_M b$ to mean $|a-b|\le M$.

\begin{lem}\label{l:diverge}
Let $g,h\in\aga$ be any nontrivial elements and let $v$ be a vertex of $\gex$.  
Suppose that the sequence $\{d(v^{g^{n}},v^{h^{n}})\}_{n\ge1}$ is bounded. 
Then the diameter of the set $\{v^{g^{n}hg^{-n}}\co n\ge1\}$ is finite.
\end{lem}
\begin{proof}
Assume that there exists $M>0$ such that $d(v^{g^{n}},v^{h^{n}})\le M$ for every $n$.
By Lemma~\ref{lem:covering distance} and the triangle inequality, for every $n$ we have
 \[M\ge  d(v^{g^{n}},v^{h^{n}})=_L d(v^{g^{n}},v^{h^{n}h})=_M d(v^{g^{n}},v^{g^{n}h})=d(v,v^{g^{n}hg^{-n}}).\]  This establishes the claim. 
\end{proof}

Observe that the proof of Lemma \ref{l:diverge} is more or less the same as the standard proof that if $g$ and $h$ are elements of a finitely generated group $G$ and if the sets $\{g^n\}$ and $\{h^n\}$ in $\cay(G)$ fail to diverge, then there exists a nonzero $N$ such that $[g^N,h]=1$ (see \cite[p.467]{BH1999}).  The complicating factor in the case of the extension graph is that $\gex$ not locally finite, so that bounded diameter subsets need not be finite.  This problem is circumvented by the following fact:

\begin{lem}\label{l:boundedconjugation}
Let $\lambda\in\aga$ be a loxodromic element and $g$ be an element of $\aga$ which is not contained in the centralizer of $\lambda$.  Then as $n$ tends to infinity, we have that \[\stl{\lambda^{n}g\lambda^{-n}}\to\infty.\] 
\end{lem}

If $\lambda'$ is a loxodromic element of $\aga$ which does not share a common nonzero power with $\lambda$ then $\lambda'$ is not contained in the centralizer of $\lambda$ and thus satisfies the hypotheses on $g$ in Lemma \ref{l:boundedconjugation}.
\begin{proof}
Note that $\stl{\lambda^n g \lambda^{-n}}$ is monotone increasing with respect to $n$.
Suppose for every $n\ge0$ we have $\stl{\lambda^n g \lambda^{-n}} < t <\infty$.
Recall our notation $B_t = \{h\in A(\Gamma)\co \stl{h}\le t\}$.
Choose $M>0$ such that $\stl{\lambda^M}=\stl{\lambda^{-M}}>2t+2$.
For every $n\ge0$ we see that $\lambda^n g \lambda^{-n} \in B_t\cap \lambda^{-M} B_t \lambda^M$.
Lemma~\ref{lem:st} implies that $\lambda^n g \lambda^{-n} = \lambda^m g \lambda^{-m}$ for some $m>n>0$.  It follows that $\lambda^{m-n}$ commutes with $\lambda^n g\lambda^{-n}$.  The classification of two--generated subgroups of right-angled Artin groups (see \cite[Theorem 1.1]{baudisch}) implies that $\lambda$ and $g$ either commute or generate a free group.  We thus have that $\lambda^{m-n}$ commutes with $\lambda^n g\lambda^{-n}$ if and only if $[\lambda,g]=1$, which contradicts our initial assumption.
\end{proof}

\begin{lem}\label{l:endpoint}
Let $\lambda\in\aga$ be loxodromic.  Then there exists a unique pair $p^{\pm}\subseteq\partial\gex$ with respect to which $\lambda$ acts by north--south dynamics.
\end{lem}
\begin{proof}
Let $v\in\gex$ be an arbitrary vertex.  Since $\lambda$ is loxodromic, we have that $d(v,v^{\lambda^n})$ tends to infinity.  Since $\gex$ is a quasi--tree, it follows that $v^{\lambda^n}$ converges to a point $p^+\in\partial\gex$.  We define $p^-$ to be the corresponding accumulation point in $\partial\gex$ for powers of $\lambda^{-1}$.

Let $\gamma$ be a quasi--geodesic ray emanating from $v$ whose endpoint $p'$ is different from $p^{\pm}$.  Then hyperbolicity implies that for all sufficiently large $N$, the quasi--geodesic $\gamma_N$ which travels from $v$ to $v^{\lambda^N}$ and then follows the corresponding translate of $\gamma$, the endpoint of $\gamma_N$ is strictly closer to $p^+$ than $p'$.
\end{proof}

We are now ready to prove Lemma \ref{l:northsouth}:
\begin{proof}[Proof of Lemma \ref{l:northsouth}]
Let $\lambda_1$ and $\lambda_2$ be non--commuting loxodromics.  By Lemma \ref{l:endpoint}, we have that for each $i$, the loxodromics $\lambda_i$ act on $\partial\gex$ by north--south dynamics.  It suffices to show that none of the pairs of endpoints $p_1^{\pm}$ and $p_2^{\pm}$ coincide.

Without loss of generality, suppose $p=p_1^+=p_2^+$.  Then for any vertex $v\in\gex$, the quasi--geodesics $\{v^{\lambda_1^n}\}$ and $\{v^{\lambda_2^n}\}$ must fellow travel and converge to $p$.    Combining Lemmas \ref{l:diverge} and \ref{l:boundedconjugation}, we see that $\lambda_1$ and $\lambda_2$ must commute, a contradiction.
\end{proof}

Theorem \ref{t:loxfree} follows immediately.

\section{Purely loxodromic subgroups of $A(\Gam)$}
In \cite{bowditchoneended} (cf. \cite{dahmanifujiwara}), Bowditch proves that if $G\le \Mod(S)$ is a finitely presented, one--ended, purely pseudo-Anosov subgroup of the mapping class group $\Mod(S)$, then there are finitely many isomorphic subgroups $\{G_1,\ldots,G_k\}$ of $\Mod(S)$ and an element $\psi\in\Mod(S)$ such that $G^{\psi}=G_i$ for some $1\leq i\leq k$.  In other words, finitely presented, one--ended, purely pseudo-Anosov subgroups of $\Mod(S)$ fall into finitely many conjugacy classes of subgroups in $\Mod(S)$ for each isomorphism type.

Finitely presented, one--ended subgroups of right-angled Artin groups are also quite restricted.  In \cite{leiningeroneended}, Leininger proves that any finitely presented, one--ended subgroup of a right-angled Artin group contains a nontrivial element whose support,  up to conjugacy, is contained in a nontrivial join.  In light of our characterization of loxodromic and elliptic elements of right-angled Artin groups, Leininger's result implies that there are no finitely presented, one ended, purely loxodromic subgroups of right-angled Artin groups.  By analogy to purely pseudo-Anosov subgroups of $\Mod(S)$, a \emph{purely loxodromic} subgroup of $\aga$ is a subgroup $G$ for which every $1\neq g\in G$ is loxodromic in $\aga$.  We will prove the following result, which completely characterizes purely loxodromic subgroups of right-angled Artin groups:

\begin{thm}\label{t:loxpure}
A purely loxodromic subgroup of a right-angled Artin group is free.
\end{thm}

Observe that we put no further hypotheses on the purely loxodromic subgroup.  In particular, we do not assume that it is finitely generated.

\begin{proof}[Proof of Theorem \ref{t:loxpure}]
Let $G$ be a purely loxodromic subgroup of $A(\Gamma)$.
The proof proceeds by induction on the number of vertices of $\Gamma$.  The crucial observation is that if $v\in V(\Gamma)$ and if $g\in \aga$ then the assumption that $G$ is purely loxodromic implies \[G\cap\langle\st(v)\rangle^g=\langle 1\rangle,\] since every cyclically reduced element of $\aga$ whose support is contained in a nontrivial join is elliptic.

The base case of the induction is where $\Gamma=v$.  In this case, $\aga$ is purely elliptic, in which case the result is vacuous.  For the inductive step, let $v\in V(\Gamma)$.  Observe that we have the following free product with amalgamation description of $\aga$: 
\[\aga\cong\langle\st(v)\rangle*_{\langle\lk(v)\rangle}A(\Gamma_v),\] where $\Gamma_v$ is the graph spanned by $V(\Gamma)\setminus v$.  By standard Bass--Serre Theory, we see that $G$ acts on the corresponding Bass--Serre tree with trivial edge stabilizer, since $G$ is purely loxodromic.  In particular, there exists a (possibly infinite) collection of subgroups $\{H_i\}$ with each $H_i$ conjugate to $A(\Gamma_v)$ in $\aga$ and $0\leq r\leq\infty$ such that \[G\le \Asterisk_i H_i * F_r.\]

Observe that if $g\in A(\Gamma_v)\le \aga$ is loxodromic in $\aga$ then $g$ is loxodromic as an element of $A(\Gamma_v)$.  Indeed, if $g$ is cyclically reduced then $\supp(g)\subseteq\Gamma_v$ is not contained in a subjoin of $\Gamma$, whence it is not contained in a subjoin of $\Gamma_v$.  Since $H_i$ is conjugate to $A(\Gamma_v)$ and $\Gamma_v$ has fewer vertices than $\Gamma$, we see that $G$ is free by induction.
\end{proof}

\section{Subsurface and vertex link projections}\label{s:proj}
In this section, we will give the necessary setup to develop the Bounded Geodesic Image Theorem and the distance formula for right-angled Artin groups.

\subsection{Distances in disconnected extension graphs}\label{ss:disc}
When $\Gamma$ is a connected graph, we have that the extension graph $\gex$ is also connected.  Equipping $\gex$ with the graph metric, it is straightforward to discuss the distance in $\gex$ between each pair of points.  If $\gex$ is disconnected, we cannot talk about paths between arbitrary pairs of vertices, but we would still like to discuss distances between vertices in those graphs.

Our approach is motivated by the curve graph of a once--punctured torus or of a four--punctured sphere.  In both of these cases, there are simple closed curves, but there are no pairs of disjoint curves.  Thus, the usual definition of the curve graph gives us a graph with infinitely many vertices and no edges.

The definition of the curve graph for these two surfaces is modified in such a way that two curves are adjacent whenever they intersect minimally.  In the case of a once--punctured torus, pairs of curves which intersect exactly once are adjacent.  In the case of the four--punctured sphere, two curves are adjacent when they intersect exactly twice.

From an algebraic point of view, recall $\Mod(S_{1,1}) = \form{a,b\vert aba = bab}$.
The simple closed curves on $S_{1,1}$ are represented as $a^g$ or $b^g$ where $g,h\in \Mod(S_{1,1})$.
Two vertices $a^g$ and $b^h$ in $\mC(S_{1,1})$ are adjacent if $a^g=a^w$ and $b^h = b^w$ for some $w\in\aga$. This is equivalent to the condition $a^g b^h a^g = b^h a^g b^h$; see \cite[Proposition 3.13]{FM2012}.
It turns out that the curve graph in this case is the Farey graph.

Now let us consider a finite discrete graph $\Gamma$ and $F$ be the free group generated by $V(\Gamma)$.
We will define $\Gamma^e$ as a graph with the vertex set $\{v^g \co v\in V(\Gamma)\text{ and }g\in F\}$
such that two vertices $u^g$ and $v^h$ are adjacent if $u^g = u^w$ and $v^h = v^w$ for some $w\in F$.  Equivalently, a pair of vertices $(u^g,v^h)$ spans an edge if this pair is conjugate via the diagonal action of some word $w\in F$ to the pair $(u,v)\in V(\Gamma)\times V(\Gamma)$.
Note that there is a simply transitive simplicial action of $F$ on $\Gamma^e$ such that all the edges are in the orbit of an edge joining two vertices of $\Gamma$; see Figure~\ref{fig:d2ext} for the case in which $\Gamma$ has two vertices. 

\begin{figure}[htb!]
  \begin{center}
\includegraphics[width=.4\textwidth]{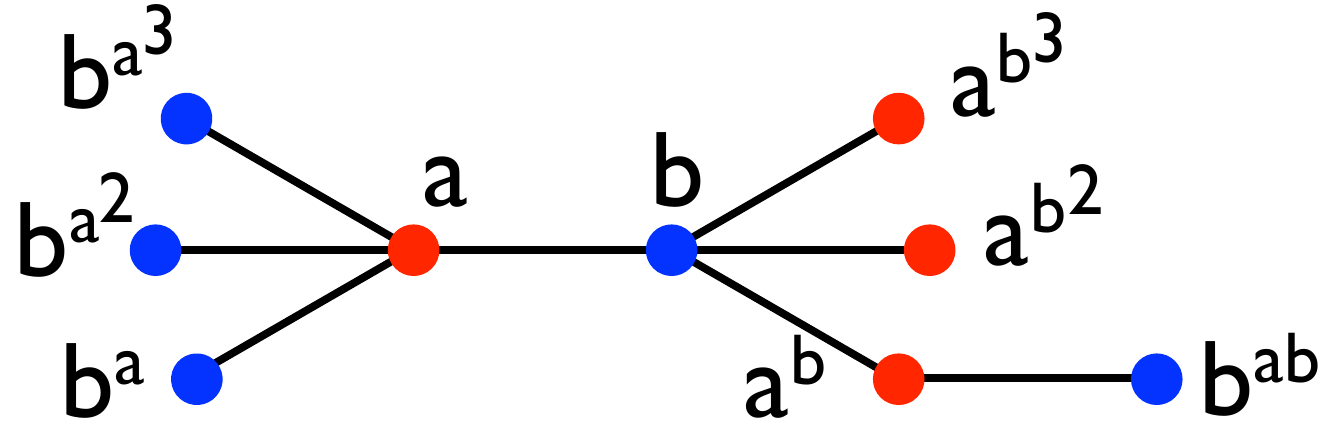}
  \caption{A part of $\Gamma^e$.}
  \label{fig:d2ext}
    \end{center}
\end{figure}

Recall that the \emph{covering distance} $d'$ as in Section~\ref{s:elec} is a well-defined metric on $\gex$ even when $\Gamma$ is not connected; see also  Section~\ref{ss:morestar}.
In the case where $\Gamma$ is connected we have that for all vertices $v$ and elements $g\in\aga$, the distance $d(v,v^g)$ is coarsely given by the star length $\stl{g}$ as shown in Lemma~\ref{lem:covering distance}. 
If $\Gamma$ is  a finite discrete graph,
then the distance on $\Gamma^e$ is precisely the covering distance.
Note that the star length and the syllable length coincide for words in free groups.
Hence, the following is obvious from Lemma~\ref{lem:covering distance} (1).

\begin{lem}\label{lem:discrete}
Let $\Gamma$ be a finite discrete graph.
Suppose $g\in A(\Gamma)$ and $a, b\in V(\Gamma^e)$
such that $a\ne b^g$.
Then 
\[
\max(1,\syl{g}-1)\le d_{\Gamma^e}(a,b^g)\le \syl{g}+1.
\]
\end{lem}

\subsection{Subsurface and vertex link projection}
We are now in a position to define subsurface projections for right-angled Artin groups.  Since the vertices of $\gex$ do not correspond to curves on a surface in any canonical way, there is some difficulty in defining a subsurface projection in the right-angled Artin group case.

Let $\gamma\in \mC (S)$ be a vertex.  The set of essential, nonperipheral simple closed curves on the subsurface $S\setminus\gamma$ forms a curve graph in its own right, and it is canonically identified with the subcomplex $\lk(\gamma)$ of $\mC (S)$.  If $\gamma'\in\mC (S)$ is another vertex, then $\gamma'$ cannot generally be viewed as a curve on $S\setminus\gamma$.  One way to take $\gamma'$ and produce a simple closed curve on $S\setminus\gamma$ which is ``most like" $\gamma'$ is as follows.    Write $G(\gamma,\gamma')$ for the set of geodesic paths between $\gamma$ and $\gamma'$ in $\mC (S)$.  Define the \emph{subsurface projection} $\pi_{\gamma}(\gamma')$ of $\gamma'$ to $S\setminus\gamma$ by setting \[\pi_{\gamma}(\gamma')=\bigcup_{\delta\in G(\gamma,\gamma')}\delta\cap\lk(\gamma),\] following Masur and Minsky (see \cite{masurminsky1}, \cite{masurminsky2}).  Thus, $\pi_{\gamma}(\gamma')$ is the set of points of $\lk(\gamma)$ which some geodesic $\delta$ meets on its way from $\gamma'$ to $\gamma$.  Of course there is an issue of how complicated such a set might be, which is the point of Bounded Geodesic Image Theorem.

We will define subsurface projection in right-angled Artin groups in a way which is analogous to curve graph case.  In the curve graph case we have that for any vertex $\gamma$, the subgraph $\lk(\gamma)$ is itself a curve graph.  An analogous fact holds for vertices in the extension graph:

\begin{lem}\label{lem:linkext}
Let $v\in\Gamma$ be a vertex.  Then there is a canonical identification \[\lk_{\gex}(v)=(\lk_{\Gamma}(v))^e.\]
\end{lem}
Here we choose $v\in\Gamma$, since then we can discuss $\lk_{\Gamma}(v)$.  If $v$ is an arbitrary vertex of $\gex$ then $v$ is conjugate to a unique vertex of $\Gamma$.
Note our convention that $\lk_{X}(v)$ also denotes the subgraph of $X$ induced by the vertex set $\lk_X(v)$.

Note that while the diameter of $\lk_{\gex}(v)$ is finite in $\gex$ by virtue of being joined with $v$, the intrinsic diameter of $\lk_{\gex}(v)$ may not be.  Indeed, for any finite graph $\Gam$ we have that $\Gam^e$ has finite diameter if and only if $\Gam$ splits as a join.  Therefore, $\lk_{\gex}(v)$ has infinite diameter unless $\lk_{\Gam}(v)$ splits as a nontrivial join.

\begin{proof}[Proof of Lemma \ref{lem:linkext}]
The inclusion \[\lk_{\gex}(v)\supseteq(\lk_{\Gamma}(v))^e\] is clear.  Suppose that $v$ is adjacent to a vertex $w$ in $\gex$.  By the definition of edges in $\gex$, we have that there is an element $g\in\aga$ such that $(v,w)^g$ is an edge in $\Gamma$.  Since conjugation by $g$ fixes $v$, we have that $\supp(g)\subseteq\st(v)$.  It follows that we can write \[g=u\cdot v^n,\] where $\supp(u)\subseteq\lk(v)$.  Finally, conjugation by $v$ fixes the star of $v$.  It follows that $w$ can be written as a conjugate of a vertex in $\lk_{\Gamma}(v)$ by an element of $\aga$ whose support is contained in $\lk_{\Gamma}(v)$.  Thus we get the reverse inclusion \[\lk_{\gex}(v)\subseteq(\lk_{\Gamma}(v))^e.\]

To see that the two links are canonically identified, take $v'\in\lk_{\Gamma}(v)$ and $g=u\cdot v^n\in\langle\st(v)\rangle$.  We identify $(v')^g\in\lk_{\gex}(v)$ with $(v')^u\in(\lk_{\Gamma}(v))^e$.  By the previous paragraph, this identification is a well--defined bijection.
\end{proof}

Now let $v,v'\in\gex$ be vertices.  Write $G(v,v')$ for the set of geodesic paths in $\gex$ which connect $v$ and $v'$.  We define the \emph{vertex link projection} of $v'$ to $v$ by \[\pi_v(v')=\bigcup_{\delta\in G(v,v')}\delta\cap\lk(v).\]  If $X\subseteq\gex$ is an arbitrary set of vertices, we can also define its vertex link projection to $v$ by \[\pi_v(X)=\bigcup_{v'\in X}\bigcup_{\delta\in G(v,v')}\delta\cap\lk(v).\]

\section{The Bounded Geodesic Image Theorem}
In this section and for the rest of this paper unless otherwise noted, we will be assuming that the graph $\Gamma$ has no triangles and no squares and is connected.  This is equivalent to the requirement that $\aga$ contains no copies of $\Z^3$ or $F_2\times F_2$ and that $A(\Gam)$ is freely indecomposable.  We will say that a subset $X\subseteq V(\gex)$ is \emph{connected} if the subgraph of $\gex$ spanned by $X$ is connected.  In this section, we will prove the following result:

\begin{thm}[Bounded Geodesic Image Theorem for right-angled Artin groups]\label{thm:bgit}
Let $\Gamma$ be a finite connected triangle-- and square--free graph.
\begin{enumerate}
\item
There exists a constant $M$ which depends only on $\Gamma$ such that if $v,v'\in\gex$ are vertices then \[\diam_{\lk(v)}\pi_v(v')\leq M.\]
\item
Let $v\in\gex$ and let $K\subseteq\gex$ be a connected subset contained outside of a closed ball of radius two about $v$.  Then \[\diam_{\lk(v)}\pi_v(K)\leq M.\]
Furthermore, there exists a constant $M'$ which depends only on $\Gamma$ such that if $v\in\gex$ and if $X\subseteq\gex$ is a geodesic segment, ray, or bi--infinite line contained in $\gex\setminus\st(v)$, then \[\diam_{\lk(v)}\pi_v(X)\leq M'.\]  
\end{enumerate}
\end{thm}

In the statement of Theorem \ref{thm:bgit}, the notation $\lk(v)$ means $\lk_{\gex}(v)$.  Theorem \ref{thm:bgit} should be compared to the Bounded Geodesic Image Theorem of Masur and Minsky:

\begin{thm}[Bounded Geodesic Image Theorem for Surfaces, \cite{masurminsky2}]
There exists a constant $M$ which depends only on $S$ such that if $Q\subseteq S$ is a proper, incompressible surface which is not a pair of pants and if $\gamma$ is any geodesic segment, ray or bi--infinite line for which $\pi_Q(v)\neq \emptyset$ for each vertex $v$ of $\gamma$, then \[\diam_{\mC (Q)}\pi_Q(\gamma)\leq M.\]
\end{thm}

In Theorem \ref{thm:bgit} (2), the additional assumption that $K$ misses a $2$--ball about $v$ is essential.
For example, let $\Gamma=C_5$, the pentagon.  Clearly, $C_5$ is triangle-- and square--free.  Fix a vertex $v\in C_5$.  Label the vertices in $\lk_{C_5}(v)$ by $x$ and $y$, and write $a$ and $b$ for the remaining vertices, with $a$ adjacent to $x$.  Write $E\subseteq C_5^e$ be the graph whose vertices are given by 
\[\{u^g\mid u\in V(C_5) \textrm{ and } g\in\langle x,y\rangle\},\] and whose edges are inherited from $C_5^e$.  As in the proof of Lemma~\ref{lem:linkext}, \[\lk_{C_5^e}(v)\subseteq E.\]

One can show that $D=E\setminus v$ is connected.  In fact, $D$ is the graph $Y_v$ which we will construct in a more general context in Subsection \ref{ss:connmodel}.  In fact, the vertices of $D$ which correspond to the conjugates of $x$ and $y$ are precisely the vertices of degree one of $D$, and all the others have infinite degree.  It follows that if we set $K$ to be the subgraph of $D$ spanned by the vertices of infinite degree, we will obtain a connected subgraph of $C_5^e$.  Furthermore, by Lemma \ref{lem:linkext} we have that $K\cap\lk_{C_5^e}(v)=\emptyset$.  Finally, consider the map $\pi_v$ restricted to $K$.  The distance from any vertex of $K$ to $v$ is exactly two.  Since $C_5^e$ is square--free, there is a unique geodesic connecting a vertex of $K$ to $v$.  Thus, the map \[\pi_v:K\to\lk_{C_5^e}(v)\] is surjective.  It follows that $\diam_{\lk(v)}\pi_v(K)$ is not finite.

\subsection{The Bounded Geodesic Image Theorem when $\Gamma$ is a tree}\label{ss:bgit tree}
The situation at hand is significantly simpler when the graph $\Gamma$ is a tree.  Recall that in this case, the extension graph $\gex$ is also a tree.  Geodesics in trees are unique, so that if $v'\in\gex$, the set $G(v,v')$ consists of exactly one path.  Thus, the diameter of the vertex link projection of $v'$ to $v$ is zero.

Now suppose that $X$ is a connected subset of $\gex$ which avoids $\st(v)$.  Choose a geodesic path $\gamma$ from $v$ to a point in $X$, and let $v_0$ be the unique point in $\gamma\cap\lk(v)$.  The connectedness of $X$ and the $0$--hyperbolicity of trees implies that any geodesic path from $v$ to a point in $X$ must also pass through $v_0$.  Thus, $\pi_v(X)=v_0$.

The requirement that $X\cap\st(v)=\emptyset$ is clearly essential, even if $X$ is assumed to just be a geodesic segment.  A geodesic segment which enters $\st(v)$ can be extended to $v$ and then extended past $v$ arbitrarily.  Since $\lk(v)$ has infinite diameter, no bound on the projection of $X$ to $\lk(v)$ can possibly hold.

\subsection{A connected model for the link of $v$}\label{ss:connmodel}
Let us fix a vertex $v$ of $\Gamma$.
A difficulty in understanding the geometry of $\lk(v)\subseteq\gex$ is that $\lk(v)$ is completely disconnected whenever $\Gamma$ has no triangles.  In this subsection, we will build a connected graph $Y_v$ which is canonically obtained from $\Gamma\setminus v$ and which is quasi--isometric to $\lk(v)$.

We first modify $\Gamma$ so that every vertex $x$ in $\lk_\Gamma(v)$ has degree larger than one in $\Gamma$. If $\deg(x)=1$, then we simply add an extra vertex $x'$ to $\Gamma$ such that the link of $x'$ is $\{x\}$. For $x\in\lk_\Gamma(v)$ we let $B_x$ be the set of vertices
 \[\{u\in V(\Gamma\setminus v)\mid d(u,x)=1\}=\lk_{\Gamma\setminus v}(x).\] 
The set $B_x$ can be thought of as the ``boundary" of $x$.  
Note that $B_x$ is a nonempty, completely disconnected graph.

For each pair of distinct vertices $x,y\in\lk_{\Gamma}(v)$ we see that $x$ and $y$ are not adjacent in $\Gamma$
and also that $x$ and $y$ have no common neighbors. 
In particular, we have $d_{\Gamma\setminus v}(x,y)>2$.
If $x\neq y$ are vertices in $\lk_{\Gamma}(v)$, we have that $B_x\cap B_y=\emptyset$ and that $B_x\cap\lk_{\Gamma}(v)=\emptyset$.

Let $Z_v$ be the graph obtained from $\Gamma\setminus v$ by declaring that whenever $z\in B_x$ and $w\in B_y$ for a pair of distinct vertices $x,y\in\lk_{\Gamma}(v)$
we have an edge $\{z,w\}$ in $Z_v$.  Note that $Z_v$ contains a canonical copy of $\lk_{\Gamma}(v)$, since we did not add any edges between vertices of $\lk_{\Gamma}(v)$.  
It follows that $\langle\lk_{\Gamma}(v)\rangle$ can be identified with a subgroup of $A(Z_v)$.  See Figure \ref{fig:link} (a) for an illustration of how $Z_v$ might look.

\begin{figure}[htb!]
  \begin{center}
\subfloat[(a)]{\includegraphics[width=.3\textwidth]{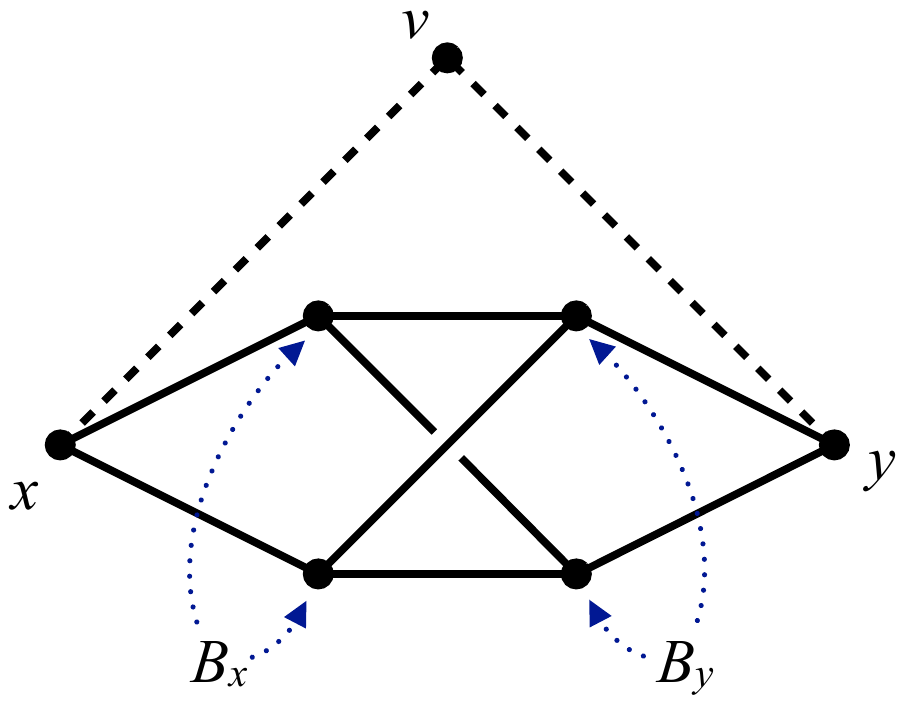}}
\hspace{.5in}
\subfloat[(b)]{\includegraphics[width=.3\textwidth]{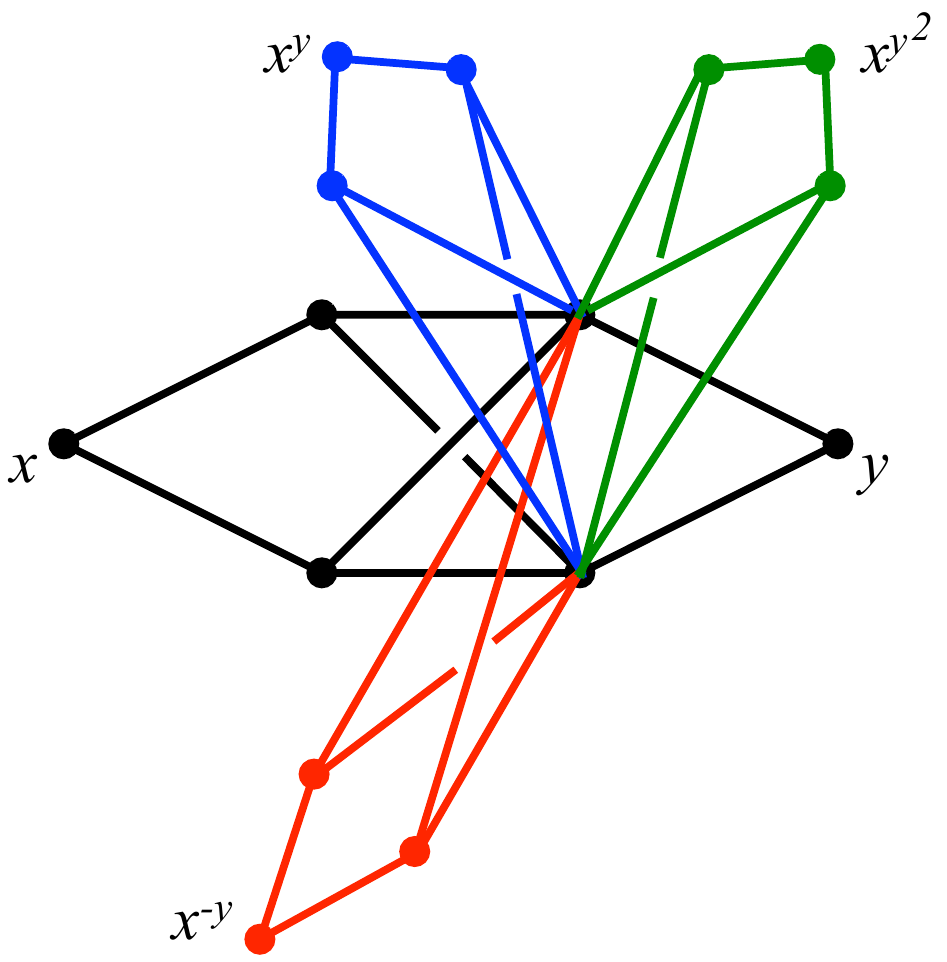}}
  \caption{(a) A schematic for the graph $Z_v$, with the vertex $v$ included.
  (b) A very small part of $Y_v$.}
  \label{fig:link}
    \end{center}
\end{figure}

The graph $Y_v$ is defined to be a subgraph of $Z_v^e$ in which we only allow conjugation by elements of $\langle\lk_{\Gamma}(v)\rangle$.  
In other words, $Y_v$ is the commutation graph of 
\[\{x^g\mid x\in V(Z_v)\textrm{ and } g\in\langle\lk_{\Gamma}(v)\rangle\}\subseteq A(Z_v).\]A small part of $Y_v$ is given in Figure \ref{fig:link} (b).

\begin{lem}\label{l:connmodel}
Let $Y_v$ be as above.
\begin{enumerate}
\item
The graph $Y_v$ is connected.
\item
The graph $Y_v$ contains a canonical copy of $\lk_{\Gamma}(v)^e=\lk_{\gex}(v)$.
\item
The inclusion of $\lk_{\Gamma}(v)^e\to Y_v$ is a quasi--isometry, where $\lk_{\Gamma}(v)^e$ is equipped with the extension graph metric and where $Y_v$ is equipped with the graph metric.
\end{enumerate}
\end{lem}
\begin{proof}
First, let us check that $Z_v$ is connected.  Suppose first that $\lk_{\Gamma}(v)$ consists of exactly one vertex.  Then $v$ has degree one in $\Gamma$.  It follows that $\Gamma\setminus v$ is connected and hence $Z_v$ is connected.  If there is a pair of distinct vertices $x,y\in\lk(v)$, then $Z_v$ contains the join $B_x*B_y$.  Furthermore since $\Gamma$ is connected, for each vertex $v'\in\Gamma\setminus v$ there is a path in $\Gamma\setminus v$ connecting $v'$ to $\st(v)$.  It follows that $Z_v$ is connected, since the image of the vertices \[\bigcup_{x\in\lk_{\Gamma}(v)} B_x\] in $Z_v$ spans a connected subgraph.

To build $Y_v$, we take one copy of $Z_v$ for each element of $g\in\langle\lk_{\Gamma}(v)\rangle\le A(Z_v)$ and label the vertices of $Z_v^g$ by the appropriate conjugate in $A(Z_v)$.  Then, we identify two vertices in two different copies of $Z_v$ if they have the same label.  One can therefore build $Y_v$ by repeatedly doubling $Z_v$ over the stars of vertices in $\lk_{\Gamma}(v)\subseteq Z_v$, as in the construction of $\gex$ by repeated doubling along stars.  It follows that $Y_v$ is connected.

The copy of $\lk_{\gex}(v)$ in $Y_v$ is given by \[\{x^g\mid x\in\lk_{\Gamma}(v) \textrm{ and } g\in\langle\lk_{\Gamma}(v)\rangle\le A(Z_v)\}.\]

To see that this inclusion is a quasi--isometry, we perform a further construction.  Let $c_v$ be the graph obtained from $Z_v$ by collapsing $B_x$ to a single vertex $b_x$ for each $x\in\lk_{\Gamma}(v)$ and by removing any other vertices not contained in $\lk_{\Gamma}(v)$ nor in \[\bigcup_{x\in\lk_{\Gamma}(v)}B_x.\]  Let $C_v\subseteq c_v^e$ be the subgraph spanned by the vertices \[\{x^g\mid x\in V(c_v)\textrm{ and }g\in\langle \lk_{\Gamma}(v)\rangle\le A(c_v)\}.\]  It is easy to see that $C_v$ is connected and quasi--isometric to $Y_v$, and that $C_v$ also contains a canonical copy of $\lk_{\gex}(v)$.

Observe that $C_v$ is a quasi--tree, and that each conjugate of a vertex $b_x$ is a cut point.  Furthermore, the vertex $x$ generates the entire centralizer of $b_x$ inside of $\langle\lk_{\Gamma}(v)\rangle$.  It follows that for each $g\in\langle\lk_{\Gamma}(v)\rangle$, we have $d_{C_v}(x,x^g)\geq \stl{g}-1$, where the star length is measured in the free group $\langle\lk_{\Gamma}(v)\rangle$.

Now write $D$ for the diameter of $c_v$.  If $x\in\lk_{\Gamma}(v)$ and $g\in\langle\lk_{\Gamma}(v)\rangle$ then $d_{C_v}(x,x^g)\leq D \stl{g}+D$ by the proof of Lemma \ref{lem:covering distance}.  It follows that the graph distance on $\lk_{\gex}(v)\subseteq C_v$ is quasi--isometric to the star length in $\langle\lk_{\Gamma}(v)\rangle$, which is what we set out to prove.
\end{proof}

There is a further slight complication resulting from the fact that conjugation by $v$ is generally nontrivial in $\aga$.  We may wish to add $v$ back to $Z_v$ and consider the subgraph $Y'_v$ of $Z_v^e$ obtained by including all the conjugates of $Z_v$ by elements of $\langle\st_{\Gamma}(v)\rangle$ instead of just $\langle\lk_{\Gamma}(v)\rangle$, and then removing the vertex $v$.  Up to quasi--isometry, the resulting graph is just $Y_v$.  Indeed, we have that conjugation by $v$ stabilizes $\lk_{\Gamma}(v)$, and in general we obtain a splitting \[\langle\st_{\Gamma}(v)\rangle\cong\bZ\times\langle\lk_{\Gamma}(v)\rangle.\]  We can view the difference between $Y'_v$ and $Y_v$ thusly: for each conjugate of $Z_v$ used to obtain $Y_v$, replace it with a $\bZ$--worth of copies of $Z_v$, which we then glue together along the corresponding copies of $\lk_{\Gamma}(v)$ sitting inside of each conjugate, like an infinite stack of pancakes identified along their boundaries.  
It is clear that the embedding $Y_v\to Y'_v$ is a quasi--isometry.

\subsection{Vertex link projection is coarsely well--defined}
In this subsection, we will prove the first part of Theorem \ref{thm:bgit}, which implies that vertex link projection is coarsely well--defined. We will need the following lemma, the proof of which can be found in~\cite[Lemma 26]{KK2012}:

\begin{lem}\label{lem:starsep}
Let \[\Gamma=\Gamma_0\subseteq\Gamma_1\subseteq\cdots\] be a filtration of graphs such that $\Gamma_{i+1}$ is obtained from $\Gamma_i$ by attaching a copy of $\Gamma$ along the star of a vertex of $v\in\Gamma$.  Suppose $x,x',z\in\Gamma_i$ are vertices such that $x$ and $x'$ lie in different components of $\Gamma_i\setminus\st_{\Gamma_i}(z)$.  Then the vertices $x$ and $x'$ lie in different components of $\gex\setminus\st_{\gex}(z)$.
\end{lem}

\begin{proof}[Proof of Theorem~\ref{thm:bgit} (1)]
If $\Gamma$ splits as a join then the triangle-- and square--free assumptions on $\Gamma$ imply that $\Gamma$ is a tree, so that $\gex$ is also a tree.  In this case, the Bounded Geodesic Image Theorem holds by the discussion we have given above.
So we assume that $\Gamma$ does not split as a join and $\gex$ has infinite diameter.  
We will also assume that $\lk_{\Gamma}(v)$ contains at least two vertices, since otherwise the conclusion is trivial.  

If $d(v,v')=1$ then $v'\in\lk_{\gex}(v)$, and if $d(v,v')=2$ then $v'\in B_x$ for some $x\in\lk_{\gex}(v)$.  By the triangle-- and square--freeness of $\Gamma$, in those two cases there is a unique geodesic connecting $v$ to $v'$, so that the vertex link projection has diameter zero.

Now we assume that $d(v,v')\geq 3$.
Let $\delta,\delta'\in G(v,v')$ be two distinct geodesics connecting $v$ and $v'$.  We will treat $\delta$ as a reference geodesic segment between $v$ and $v'$, and we will show that if $\delta'\cap\lk(v)$ is ``too far" from $\delta\cap\lk(v)$ then $\delta'$ could not have been a geodesic segment.

Write $x$ and $x'$ for the points of $\lk(v)$ which are met by $\delta$ and $\delta'$ respectively, and consider the vertices $x$ and $x'$ as vertices in $Y_v$.  Since $Y_v$ is quasi--isometric to $\lk_\gex(v)$, the distance between $x$ and $x'$ is comparable in $\lk_\gex(v)$ and in $Y_v$.  
Observe that if $d_{Y_v}(x,x')\geq 3\diam(Z_v)=M$ then there is a vertex $z\in\lk(v)$ such that $\st_{Y_v}(z)$ contains neither $x$ nor $x'$ and such that $\st_{Y_v}(z)$ separates $x$ from $x'$; see the proof 
of~\cite[Lemma 26]{KK2012}.
Figure \ref{fig:link3} gives an illustration of the fact that such a vertex $z$ exists. In this figure, the star of $x$ separates $x^y$ from $x^{y^x}$.  The boundaries of vertices in $\lk_\gex(v)$ have been collapsed to single (unlabelled) vertices.
Actually, if $d_{Y_v}(x,x')\geq N\diam(Z_v)$ then there are at least $N/3$ such $z$'s and this fact will be used in the proof of the second part of Theorem~\ref{thm:bgit}.

\begin{figure}[htb!]
  \begin{center}
\includegraphics[width=.3\textwidth]{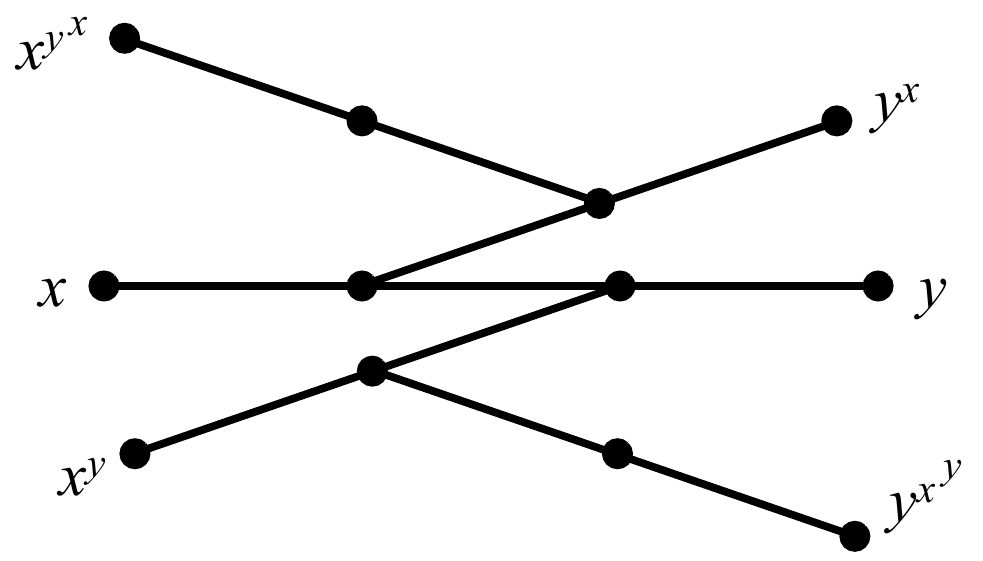}
  \caption{Distant vertices in $Y_v$ are separated by stars of intermediate vertices.}
  \label{fig:link3}
    \end{center}
\end{figure}

We claim that $\st_{\gex}(z)$ separates $x$ from $x'$ in $\gex$, and that this suffices to prove the result.  
For the latter claim, suppose that the star of $z$ separates $x$ from $x'$ in $\gex$.
The following is clear from triangle-- and square--freeness: 
\[\delta\cap\st_{\gex}(z)=\delta'\cap\st_{\gex}(z)=v.\]  
But then the subpath of $\bar\delta\cdot\delta'$ from $x$ to $x'$ is contained in $\gex\setminus\st(z)$ and we have a contradiction.

So, it suffices to see that $\st_{\gex}(z)$ separates $x$ from $x'$ in $\gex$.  Let $L_v$ be the subset of $\gex$ given by doubling over stars of vertices which lie in $\st(v)$.  Then $L_v$ can be filtered by finite subgraphs in a way which satisfies the hypotheses of Lemma \ref{lem:starsep}.  So, it suffices to see that $x$ and $x'$ are separated by $\st_{L_v}(z)$.  Observe that there is a natural map of graphs $L_v\setminus v\to Y'_v$.  The star of $z$ in $Y'_v$ separates $x$ from $x'$, whence it follows that $\st_{L_v}(z)$ separates $x$ from $x'$.
\end{proof}

\subsection{The Bounded Geodesic Image Theorem}
We now give the proof of the second part of the Bounded Geodesic Image Theorem (Theorem \ref{thm:bgit}).

\begin{proof}[Proof of Theorem~\ref{thm:bgit} (2)]
Since $K$ remains entirely outside of a $2$--ball containing $v$, we have that if $w\in\lk_{\gex}(v)$ then $\st(w)\cap K=\emptyset$.  Let $x$ and $y$ be vertices of $K$, and let $\gamma$ be a path in $K$ connecting $x$ and $y$.  Write $\delta_x$ and $\delta_y$ for geodesics in $\gex$ connecting $v$ with $x$ and $y$ respectively.  Assume that $d_{\lk_{\gex}(v)}(\pi_v(x),\pi_v(y))>M$, where $M$ is the constant given by part (1) of Theorem \ref{thm:bgit}.  By the proof of part (1) of Theorem \ref{thm:bgit}, there is a vertex $z\in\lk_{\gex}(v)$ whose star separates $w_x=\delta_x\cap\lk_{\gex}(v)$ from $w_y=\delta_y\cap\lk_{\gex}(v)$.  However, following $\delta_x$ followed by $\gamma$ followed by $\delta_y$ furnishes a path in $\gex\setminus\st(v)$ between $w_x$ and $w_y$, a contradiction.

Now suppose that $\delta\subseteq\gex\setminus\st(v)$ is a geodesic.  Write $S_2(v)$ for the sphere of radius two about $v$.  We may assume that $\delta\cap S_2(v)\neq\emptyset$, since otherwise we are in the case considered in the previous paragraph.  Observe that $\diam S_2(v)\leq 4$.  Since $\delta$ is a geodesic, we have that $|\delta\cap S_2(v)|\leq 5$.  Thus, there are at most five vertices in $\lk_{\gex}(v)$ whose links meet $\delta$.  Write $M'=18\diam(Z_v)$. 
Let $x,y\in\delta$ project to points $w_x$ and $w_y$ in $\lk_{\gex}(v)$ which are at a distance exceeding $M'$ from each other.  Then there is a vertex $z\in\lk_{\gex}(v)$ such that $\lk_{\gex}(z)$ does not meet $\delta$ and such that $\st_{\gex}(z)$ separates $w_x$ from $w_y$.  This is again a contradiction, so that the diameter of the projection of $\delta$ is at most $M'$.
\end{proof}

\subsection{Failures of the Bounded Geodesic Image Theorem}
The failure of Theorem~\ref{thm:bgit} for right-angled Artin groups containing $F_2\times F_2$ can already be seen with $F_2\times F_2$ itself.  Label the square $C_4$ cyclically by $\{a,b,c,d\}$.  It is an easy exercise to check that $C_4^e\cong K_{\infty}*K_{\infty}$, the complete bipartite graph on two countable discrete sets of vertices.  One of the copies of $K_{\infty}$ contains the conjugates of $a$ and $c$, whereas the other contains the conjugates of $b$ and $d$.  Observe that since $a$ and $c$ do not commute with each other, the distance between them in $C_4^e$ is at least two.  However, any length two path from $a$ to $c$ which passes through any conjugate of $b$ or $d$ realizes a geodesic path from $a$ to $c$.  Thus, $\pi_{a}(c)$ consists of $\lk_{C_4^e}(a)$, which has infinite diameter in the intrinsic metric on the link of $a$.

It is possible that Theorem~\ref{thm:bgit} holds for square--free graphs which have triangles.  However, the analysis we have carried out above fails, as the structure of links in $\gex$ becomes much more complicated in this situation.

\section{The distance formula}\label{s:distance}
We have seen already that distance in $\gex$ is coarsely just star length in the right-angled Artin group $\aga$.  In this section, we will develop a Masur--Minsky style distance formula, which will incorporate the extra data that vertex link projections give us in order to refine distance estimates in $\gex$.

The Masur--Minsky distance formula allows one to use distances in the curve graph to coarsely recover distances in the mapping class groups.  The high degree of symmetry of the extension graph does not allow such a result to be proved for right-angled Artin groups.  Indeed, the automorphism group of the curve graph is, up to finite index, the mapping class group of the underlying surface.  A complicating factor in the extension graph case is that powers of vertices are indistinguishable.  To see this, consider a vertex $v\in V(\Gamma)=\Gamma_0$.  In $\gex$, there are copies of $\Gamma$ meeting $\Gamma_0$ along the star of $v$ which correspond to conjugates of $\Gamma$ by powers of $v$.  Let $v_0$ be another vertex of $\Gamma$ which does not lie in the star of $v$.  We will see (Theorem \ref{thm:aut}) that there are symmetries of $\gex$ which fix $\Gamma_0$ and which permute the $v$--conjugates of $\Gamma$ arbitrarily.  In particular, the extension graph cannot distinguish between the conjugates $\{v_0^{v^k}\}_{k\in\bZ}$.  It thus seems unlikely that the geometry of the extension graph can be used to recover word length in $\aga$.

Our generalization of the Masur--Minsky distance formula to the right-angled Artin case will show that the geometry of $\gex$ can be used to recover distance in $\aga$ up to forgetting powers of the vertex generators.  In other words, we will be recovering the syllable length of words in $\aga$.  

\subsection{Markings for mapping class groups}
The Masur--Minsky distance formula coarsely measures distances in the Cayley graph of the mapping class group $\Mod(S)$, using only the data of the curve graph $\mC(S)$:

\begin{thm}[{\cite[Theorem 6.12]{masurminsky2}}]\label{thm:mmdistformula}
Let $S$ be a surface.  There are constants $M_0,C>0$ and a constant $K>1$ such that for each pair $\mu,\nu$ of complete, clean markings on $S$ and for all $M\geq M_0$, we have \[\frac{1}{K}d_{\mathcal{M}(S)}(\mu,\nu)-C\leq \sum_{S_0\subseteq S,d_{S_0}(\mu,\nu)\geq M}d_{S_0}(\mu,\nu)\leq Kd_{\mathcal{M}(S)}(\mu,\nu)+C.\]
\end{thm}

For the convenience of the reader, we will briefly summarize the terminology of Theorem \ref{thm:mmdistformula}.  The object $\mathcal{M}(S)$ is the \emph{marking graph} of $S$.  A \emph{marking} $\mu$ is a finite set of ordered pairs $\{(\alpha_i,\beta_i)\}$, where the curves $\{\alpha_i\}$ form a multicurve on $S$ and the $\{\beta_i\}$ are \emph{transversals} of the corresponding curves $\{\alpha_i\}$, which is to say $\beta_i$ is either empty or is a subset of the annular curve graph $\mC(\alpha_i)$ of diameter one.  A transversal $\beta$ in the pair $(\alpha,\beta)$ is \emph{clean} if there is a curve $\gamma\subseteq S$ such that $\beta$ is the projection of $\gamma$ to the annular surface filled by $\alpha$, and such that $\alpha$ and $\gamma$ are Farey--graph neighbors in the curve graph of $\Fill(\gamma\cup\alpha)$ (which is to say they intersect a minimal nonzero number of times).  A marking is \emph{clean} if every transversal is clean and the curve $\gamma_i$ inducing the transversal $\beta_i$ does not intersect any $\alpha_j$ other than $\alpha_i$.  The marking is \emph{complete} if the curves $\{\alpha_i\}$ form a pants decomposition of $S$ and none of the transversals are empty.

If a marking is complete and clean, a transversal $\beta_i$ determines the curve $\gamma_i$ uniquely.  Masur and Minsky show that if a marking is not clean, then in some appropriate sense, there is a bounded number of ways to obtain a \emph{compatible} clean marking.

The marking graph $\mathcal{M}(S)$ is the graph whose vertices are complete, clean markings on $S$.  Two markings are connected by an edge if they differ by one of two elementary moves, called \emph{twists} and \emph{flips}.  To define these moves, let $\mu=\{(\alpha_i,\pi_{\alpha_i}(\gamma_i))\}$ and $\nu=\{(\alpha_i',\pi_{\alpha_i'}(\gamma_i'))\}$ be two markings.  We say that $\mu$ and $\nu$ differ by a twist if for some $\{(\alpha_j,\pi_{\alpha_j}(\gamma_j))=(\alpha_j',\pi_{\alpha_j'}(\gamma_j'))\}$ for all $j\neq i$, and $\{\pi_{\alpha_i'}(\gamma_i')\}$ is obtained from $\{\pi_{\alpha_i}(\gamma_i)\}$ by a twist or half--twist about $\alpha_i$.  The unclean marking $\nu'$ differs from $\mu$ by a flip if for some $i$, we have $(\alpha_i',\pi_{\alpha_i'}(\gamma_i'))=(\gamma_i,\pi_{\gamma_i}(\alpha_i))$.  The marking $\nu'$ is then replaced by a clean marking $\nu$ which is compatible with $\nu'$.

Masur and Minsky show that the marking graph is locally finite and admits a natural proper, cocompact action of $\Mod(S)$ by isometries.  It follows that $\Mod(S)$ is quasi--isometric to $\mathcal{M}(S)$ via the orbit map.

Finally, if $S_0\subseteq S$ is an essential subsurface, we define the projection $\pi_{S_0}(\mu)$ of a complete marking $\mu$ on $S$.  If $S_0$ is not the annular subsurface filled by some $\alpha_i$, we set $\pi_{S_0}(\mu)=\pi_{S_0}(\{\alpha_i\})$.  If $S_0$ is the annular subsurface filled by $\alpha_i$, we define $\pi_{S_0}(\mu)=\beta_i$.  Since $\mu$ is a complete marking, the projection is always nonempty.

One interpretation of the distance formula is as follows: projections of curves to subsurfaces of $S$ allow us to triangulate locations of markings in $\mathcal{M}(S)$, and thus via the orbit map quasi--isometry, to better understand the geometry of $\Mod(S)$.  Thus, subsurface projections give a good coarse notion of visual direction of $\mC(S)$.  We would like to translate this interpretation to the setting of $\gex$ and $\aga$, with the idea that distances between points in $\lk_{\gex}(v)$ give good coarse notions of visual angle between two geodesics emanating from $v$ and passing through those two points.

\subsection{The Weil--Petersson metric and projections of the pants complex}
In the analogy between right-angled Artin groups and mapping class groups we have that a vertex of $\Gamma$, viewed as a generator of $\aga$, corresponds to (a power of) a Dehn twist about a simple closed curve on a surface $S$.  We mentioned above that our distance formula, we can only recover distances in $\aga$ after forgetting the exponents of generators used to express elements of $\aga$.  In the Masur--Minsky distance formula, one can consider projections to a restricted class of subsurfaces instead of to all subsurfaces.  If one ``ignores Dehn twists" in the sense that one excludes projections to annular subsurfaces, one coarsely recovers the Weil--Petersson metric on Teichm\"uller space.  This follows from combining results of Masur--Minsky and Brock.  Before we state those results, we recall some definitions.  A \emph{pants decomposition} of a surface $S$ is the zero--skeleton of a maximal simplex in the curve graph of $S$.  The \emph{pants graph} of $S$, written $\mathbb{P}(S)$, is defined to be the graph whose vertices are pants decompositions of $S$ and whose edges are given by ``minimal intersection".  Precisely, $P_1$ and $P_2$ are adjacent in the pants graph if $P_1$ and $P_2$ agree on all but one pair of curves, say $c_1\in P_1$ and $c_2\in P_2$, if $(P_1\setminus c_1)\cup c_2$ is a pants decomposition, and if $c_1$ and $c_2$ intersect minimally.  One often says that $P_1$ and $P_2$ differ by an \emph{elementary move}.

\begin{thm}[Masur--Minsky (\cite{masurminsky2}, Theorem 6.12)]
There is a constant $M_0=M_0(S)$ so that for all $M\geq M_0$, there exist constants $K$ and $C$ such that if $P_1$ and $P_2$ are pants decompositions in the pants complex $\mathbb{P}(S)$, we have \[\frac{1}{K}d_{\mathbb{P}(S)}(P_1,P_2)-C\leq \sum_{S_0\subseteq S} d_{S_0}(P_1,P_2)\leq K d_{\mathbb{P}(S)}(P_1,P_2)+C,\] where the sum is taken over all non--annular essential subsurfaces of $S$ for which $d_{S_0}(P_1,P_2)\geq M$.
\end{thm}

In different notation, we will write the estimate in the previous result as 
\[\sum_{S_0\subseteq S} d_{S_0}(P_1,P_2)\asymp_{K,C} d_{\mathbb{P}(S)}(P_1,P_2).\]  
In the middle term of the estimate, we view pants decompositions as subsets of the curve graph of $S$ and define \[d_{S_0}(P_1,P_2)=d_{S_0}(\pi_{S_0}(P_1),\pi_{S_0}(P_2)).\]

Thus, projections of pants decompositions to non--annular subsurfaces coarsely recovers distance in the pants graph.  The pants graph is quasi--isometric to Teichm\"uller space with the Weil--Petersson metric:

\begin{thm}[Brock (\cite{brockconvexcores}, Theorem 1.1)]
The pants graph $\mathbb{P}(S)$ is quasi--isometric to Teichm\"uller space with the Weil--Petersson metric, via the map which takes a pants decomposition $P$ to the maximally noded surface given by collapsing the components of $P$.
\end{thm}

In this sense, the syllable length in a right-angled Artin group can be thought of as an analogue of the Weil--Petersson metric on Teichm\"uller space.

\subsection{Free groups}
If $D_n$ is a completely disconnected graph, we have that the link of any vertex in $D_n$ is empty.  In particular, the notions of vertices and stars of vertices coincide, so that star length and syllable length coincide.  If $v$ and $v^g$ are vertices of $D_n^e$, we have that $\pi_{\lk(v)}(v^g)=\emptyset$.  Thus, the distance formula for a free group is vacuous.

\subsection{Right-angled Artin groups on trees}
The first non--vacuous case of the distance formula is in the case where the graph $\Gamma$ is a tree. In this case, the extension graph $\gex$ is also a tree.  Since any two vertices in a tree are connected by a unique geodesic we have that all but finitely many projections of a finite geodesic segment are nonzero, so that one need not discard any projections:
\begin{prop}\label{p:disttree}
Let $\Gamma$ be a tree and $D=\diam(\Gamma)$. If $g\in\aga$, then we have that
 \[\sum_{x\in V(\gex)}d_{\lk(x)}(\pi_x(v),\pi_x(v^g))\asymp_{4D,4D} \syl{g}\] 
 for some $v\in V(\Gamma)$.
\end{prop}
Recall the notation $\asymp_{K,C}$ denotes $K,C$--quasi--isometry. If $x$ and $y$ are two points of a tree, let us denote by $[x,y]$ the unique geodesic joining the two points.

\begin{proof}[Proof of Proposition \ref{p:disttree}]
We may assume $D\ge2$. Put $k = \stl{g}$.
By Lemma~\ref{lem:syl}, we can write $g = h_k h_{k-1}\cdots h_1$ such that $h_i\in\form{\st(y_i)}$ for some $y_i\in V(\Gamma)$ and
$\syl{g}=\sum_i \syl{h_i}$. If $\syl{h_i}=1$, then we may further assume $\supp(h_i)=\{y_i\}$.
Let us first consider the case that there exists $v\notin\st(y_1)\cup\st(y_k)$.

Let us denote by $\gamma$ the unique geodesic in $\gex$ connecting $v$ to $v^g$.
For $x\in \Gamma^e$, put 
\[\delta(x) = d_{\lk(x)}(\pi_x(v),\pi_x(v^g)).\] 
Only the vertices on $\gamma$ contribute to the sum $\sum_{x\in V(\gex)}\delta(x)$;
see Section \ref{ss:bgit tree}.

For simplicity, let $g_i = h_ih_{i-1}\cdots h_1$ for $i=1,2,\ldots,k$ and $g_0=1$.
We define 
\[c_0 = [v,y_1],c_k=[y_k^g,v^g], c_i=[y_i^{g_{i}},y_{i+1}^{g_{i+1}}]\text{ for }i=1,2,\ldots,k-1.\]
The concatenation $C=c_0\cdot c_1\cdots c_k$ is a path from $v$ to $v_g$.
In particular, we have $\gamma\subseteq C$.
Since $y_{i+1}^{g_{i+1}} = y_{i+1}^{h_{i+1} g_i} = y_{i+1}^{g_i}$, we have $c_i\subseteq\Gamma^{g_i}$.
The minimality of $k$ and Lemma~\ref{lem:star} (4) imply that $c_i\cap c_j=\varnothing$ for $i+1<j$.
Using the assumption that $v\not\in \st(y_1)\cup\st(y_k)$, we deduce that
 $\gamma = c'_0\cdot c'_1\cdots c'_k$ for some nontrivial geodesic segments $c'_i\subseteq c_i$.
Put $c'_{i-1}\cap c'_i = z_i^{g_{i-1}}$ where $z_i\in\Gamma$.
For each $i$, let us denote the length--two subpath of $\gamma$ centered at $z_i^{g_{i-1}}$ as
 $(a_i^{g_{i-1}},z_i^{g_{i-1}},b_i^{g_i})$.
By Lemma~\ref{lem:star}, either $z_i=y_i$ or $z_i\in \lk_\Gamma(y_i)$
so that we can write $h_i = z_i^{m_i} w_i$ for some $w_i\in \form{\lk(z_i)}$.
Note that $\syl{w_i}\le \syl{h_i}\le \syl{w_i}+1$.
In either of the two cases, $\delta(z_i^{g_{i-1}})$ is the distance between $a_i$ and $b_i^{h_i} = b_i^{w_i}$ 
in $\lk(z_i)^e$. 
Lemma~\ref{lem:discrete} implies that 
\[\syl{h_i}/3\le \max(1,\syl{w_i}-1)\le \delta(z_i^{g_{i-1}})\le \syl{w_i}+1\le \syl{h_i}+1.\]

Let $B$ be the set of the interior vertices of $\gamma$ which is not of the form $z_i^{g_{i-1}}$ for $i=1,2,\ldots,k$.
Then $\delta(x) = 1$ for $x\in B$ since the length--two subpath of $\gamma$ centered at $x$ is contained in a conjugate of $\Gamma$.
Lemma~\ref{lem:covering distance} (3) implies that 
\[ 0\le |B| \le  ( d(v,v^g) - 1) - (k-2) = d(v,v^g) -k+1 \le (D-1) k+D+1.\]
From $\stl{g}\le\syl{g}$ we see that
\[
\syl{g}/3\le\sum_{x\in\Gamma^e}\delta(x) = \sum_{i=1}^k \delta(z_i^{g_{i-1}}) + |B|\le(D+1)(\syl{g}+1).
\]
Now let us assume $V(\Gamma)=\st(y_1)\cup\st(y_k)$. We may still choose $v\not\in\{y_1,y_k\}$.
If we define $c'_i$ and $z_i$ as before, it is then possible that $v=z_1$ or $v^g=z_k^{g_{k-1}}$; this occurs only if $\syl{h_1}$ or $\syl{h_k}$ is at most two, respectively. Hence,
\[
(\syl{g}-4)/3 \le\sum_{x\in\Gamma^e}\delta(x) \le(D+1)(\syl{g}+1).
\]
\end{proof}

\subsection{The general triangle-- and square--free case}
We would now like to establish a more general version of the distance formula.  
We let $\Gamma$ be a triangle-- and square--free graph
such that $D=\diam(\Gamma)\ge2$.

\begin{lem}\label{l:sylgen}
Let $g\in\aga$ satisfy $\stl{g}=1$.  Then there is a vertex $v\in V(\Gamma)$ such that for every geodesic $\gamma$ from $v$ to $v^{g}$ we have \[\sum_{x\in V(\gamma)}d_{\lk(x)}(\pi_x(v),\pi_x(v^g))\asymp_{1,1} \syl{g}.\]
\end{lem}
\begin{proof}
We may assume $\Gamma$ is not a star, since otherwise the conclusion is clear.
Choose $y\in V(\Gamma)$ such that
 $\supp(g)\subseteq\st(y)$ and $\supp(g)\not\subseteq \{z\}$ for every $z\ne y$.
We will consider the following three cases. 

\textit{Case 1.} $\supp(g)=\{y\}$.

We have $\Gamma\cap\Gamma^g =\st(y)$.
Let $(y,z,v)$ be a length-two path in $\Gamma$. Then 
$v\ne v^g\in\lk(z)$
and $d_{\lk(z)}(v,v^g) = \syl{g}=1$.

\textit{Case 2.} $\supp(g)=\{y,a\}$ for some $a\in\lk(y)$.

We may assume $\deg(y)>1$ for otherwise, we can switch the roles of $a$ and $y$.
Put $z=y$ and choose $v\in(\lk(z)\setminus\{a\})\cap \Gamma$. Then $v\ne v^g\in\lk(z)$ and $d_{\lk(z)}(v,v^g) = 1 =  \syl{g}-1$.

\textit{Case 3.} $|\supp(g)\cap\lk(y)|\ge2$.

Put $z=y$ and choose $v\in\lk(z)\cap \Gamma$. Then $v\ne v^g\in\lk(z)$ and $\syl{g}-1\le d_{\lk(z)}(v,v^g) \le \syl{g}$.

In the above three cases, $\gamma=(v,z,v^g)$ is the unique geodesic from $v$ to $v^g$.
Hence, the given sum consists of only one term and coarsely coincides with $\syl{g}$ using the quasi-isometry constants $K=1=C$.
\end{proof}

Lemma \ref{l:sylgen} lends itself to a general distance formula for triangle-- and square--free right-angled Artin groups.

\begin{prop}\label{p:sylgen}
For all $1\neq g\in \aga$, there exists a $(10D,10D)$--quasi--geodesic $\gamma\co[0,\ell]\to \gex$ parametrized by arc length such that
$\gamma(0)=v,\gamma(\ell)=v^g$ for some $v\in V(\Gamma)$ and 
\[\sum_{i=2}^{\ell-1}d_{\lk(\gamma(i))}(\gamma(i-1),\gamma(i+1))\asymp_{10D,10D} \syl{g}.\]
\end{prop}

\begin{proof}
Let $g = h_k h_{k-1}\cdots h_1$ such that $\supp(h_i)\subseteq \st(y_i)$ and $\syl{g}=\sum_i\syl{h_i}$.
As in the proof of Lemma~\ref{l:sylgen}, we may further assume the following.
\begin{enumerate}[(i)]
\item
If $\syl{h_i}=1$, then $\supp(h_i)=\{y_i\}$.
\item
If $\syl{h_i}\ge2$, then $\deg(y_i)>1$.
\end{enumerate}
Put $g_i = h_i h_{i-1}\cdots h_1$ and $g_0=1$.

We choose $z_i$ and $v_i$ as in the proof of Lemma~\ref{l:sylgen}.
Namely, if $\syl{h_i}=1$ then we choose a length-two path $(y_i,z_i,v_i)$ in $\Gamma$.
If $\syl{h_i}\ge2$, then we set $z_i = y_i$ and choose $v_i\in\lk(z_i)$ such that $v_i^{h_i}\ne v_i$.
Note that $z_i^{h_i}=z_i$. 
In (i) above, we see that $d_{\lk(z_i)}(v_i,v_i^{h_i}) = \syl{h_i}=1$.
In (ii), we note that $\syl{h_i}/2\le \max(1,\syl{h_i}-1)\le d_{\lk(z_i)}(v_i,v_i^{h_i}) \le \syl{h_i}$.

We put $v = v_1$.
Choose a shortest path $c_i$ in $\Gamma$ from $v_i$ to $v_{i+1}$ for $i=1,2,\ldots,k-1$
and let $c_k$ be a shortest path in $\Gamma$ from  $v_k$ to $v$.
Define $\gamma$ to be the concatenation 
\[\gamma=(v_1,z_1,v_1^{g_1})\cdot c_1^{g_1}\cdot(v_2^{g_1},z_2^{g_1},v_2^{g_2})\cdot c_2^{g_2}
\cdots
(v_k^{g_{k-1}},z_k^{g_{k-1}},v_k^{g_k})\cdot c_k^g.\]
Let $\gamma\co[0,\ell]\to \gex$ be the parametrization by arc length.
In particular, $\gamma([0,\ell]\cap\mathbb{Z})$ is contained in the vertex set of $\gex$.
Put $A = \{i \co \gamma(i) = z_j^{g_{j-1}} \text{ for some }j\}$ and $B = \{1,2,\ldots,\ell-1\}\setminus A$.

From the above observations, we have
\[\syl{g}/2 \le \sum_{i\in A} d_{\lk(\gamma(i))}(\gamma(i-1),\gamma(i+1)) \le \syl{g}.\]

If $i\in B$, then $d_{\lk(\gamma(i))}(\gamma(i-1),\gamma(i+1)) \le 1$ since $\{\gamma(i-1),\gamma(i),\gamma(i+1)\}\subseteq \Gamma^{g_j}$ for some $j$. So,
\[
0 \le \sum_{i\in B}d_{\lk(\gamma(i))}(\gamma(i-1),\gamma(i+1))
\le
(D+2)k\le (D+2)\syl{g}.
\]
Hence,
\[\sum_{i=2}^{\ell-1}d_{\lk(\gamma(i))}(\gamma(i-1),\gamma(i+1))\asymp_{D+3,0} \syl{g}.\]

It remains to show that $\gamma$ is a $(10D,10D)$--quasi--geodesic.
Let us set 
\[\delta_0=[v_1,z_1],\]
\[\delta_i = [z_i^{g_{i-1}},v_i^{g_i}]\cdot c_i^{g_i}\cdot
[v_{i+1}^{g_i},z_{i+1}^{g_i}]\subseteq \Gamma^{g_i}
\text{ for }i=1,2,\ldots,k-1,\]
\[
\text{ and }
\delta_k=[z_k^{g_{k-1}},v_k^g]\cdot c_k^g.\]
Then we have $\gamma = \delta_0 \cdot \delta_1\cdot \delta_2\cdots\delta_{k-1}\cdot  \delta_k$.
Note that the length of $\delta_i$ is between one and $D+2$.

We claim that for
 $0\le i<j\le k$ and
  $x\in\delta_i, y\in\delta_j$ we have $j-i-4\le d'(x,y)\le j- i+1$.
The upper bound is clear.
Assume $d'(x,y)\le j-i-5$.
Since $v_s^{g_s}\in\delta_s\subseteq \Gamma^{g_s}$ for each $s$
we have that $d'(v_i^{g_i},v_j^{g_j})\le j-i-3$.
Then we have $d'(v_i^{g_i},v_i^{g_j})\le j-i-2$.
On the other hand, Lemma~\ref{lem:covering distance} (1) implies that $d'(v_i^{g_i},v_i^{g_j})\ge \stl{g_j g_i^{-1}}-1=j-i-1$. This is a contradiction.

Now fix $0\le p<q\le \ell$. Let $\gamma(p)\in \delta_i$ and $\gamma(q)\in\delta_j$ where $0\le i<j\le k$.
We have that $d'(\gamma(p),\gamma(q))\ge j-i-4$. 
From the estimates on the lengths of $\delta_i$'s, we have 
\[j-i-1 \le q-p \le (D+2)(j-i+1).\]
By Lemma~\ref{lem:covering distance}, we have
\[
j-i-4\le d'(\gamma(p),\gamma(q)) \le d(\gamma(p),\gamma(q)) \le D d'(\gamma(p),\gamma(q)) \le D (j-i+1).\]
It follows that 
\begin{eqnarray*}
\frac1D d(\gamma(p),\gamma(q))-2&\le& j-i-1 \le q-p \\ 
&\le& (D+2)(j-i+1)\le (D+2)(d(\gamma(p),\gamma(q))+5).
\end{eqnarray*}
\end{proof}
The general distance formula we have given in Proposition \ref{p:sylgen} is somewhat unsatisfying in that it is not a direct analogue of the Masur--Minsky distance formula.  Ideally, the sum on the left hand side would be a sum over all vertices within $\gex$, and the terms in the summation would be projected distances which exceed a certain lower threshold.  At the time of the writing there are several issues which impede the authors from proving such a general result, perhaps the most serious of which is the lack of an appropriate notion of a tight geodesic.

\section{Automorphism group of $\Gamma^e$}
By a result of Ivanov in \cite[Theorem 1]{ivanov}, the automorphism group of the curve graph of a non--sporadic surface is commensurable with the mapping class group of the surface.  Bestvina asked the authors whether an analogous result holds for the extension graph.  The answer is no:

\begin{thm}\label{thm:aut}
The isomorphism group of an extension graph of a connected, anti--connected graph contains the infinite rank free abelian group.  In particular, $\aut(\Gam^e)$ is uncountable.
\end{thm}
Here, a graph is \emph{anti-connected} if its complement graph is connected.
\begin{proof}
Let $\Gam$ be a connected and anti--connected graph.
We observe that each star of a vertex $v_0$ separates $\gex$ into infinitely many components of infinite diameter, all of which are isomorphic to each other.  To see this, fix a vertex $v\in V(\Gamma)$.  By assumption, $\st(v)\neq \Gamma$.  Observe that for each $n>0$, we can build the graph \[\Gamma_n=\bigcup_{\st(v),i=1}^n\Gamma,\] which is obtained by identifying $n$ copies of $\Gamma$ along the star of $v$.  The graph $\Gamma_n$ naturally embeds in $\gex$.  Observe that $\Gamma_n\setminus\st(v)$ separates $\Gamma_n$ into $n$ components, so that $\st(v)$ separates $\gex$ into at least $n$ components.  It follows that $\gex\setminus\st(v)$ has infinitely many components.  Observe that the conjugation action of $v$ permutes these components.  In particular, $\gex\setminus\st(v)$ has infinitely many isomorphic components.  Since $\Gamma$ is anti--connected, we have that $\gex$ has infinite diameter, so that at least one component of $\gex\setminus\st(v)$ has infinite diameter.  By applying the conjugation action of $v$, we have that infinitely many components of $\gex\setminus\st(v)$ have infinite diameter.

Pick such an infinite component $\Gam_0$ and consider $C_0=\{\Gam_0^{v_0^i}\co i \in\Z\smallsetminus\{0\}\}$. We let $f_0$ be the conjugation action by $v_0$ on $\Gam^e$.
Fix an element in $C_0$, we find the star of another vertex $v_1$ that separates out an infinite component. Around this vertex, there is an action $f_1$ similar to the previous step fixing $\st(v_1)$.
Continuing this process, we see that \[\{f_1,f_2,\ldots\}\cong\prod_{i=1}^{\infty}\Z,\] whence the result.
\end{proof}
\section*{Acknowledgements}
We thank J. F. Manning for valuable suggestions regarding Theorem~\ref{thm:starqiclass}.  We also thank Y. Algom--Kfir, M. Bestvina, A. Hadari, C. McMullen, Y. Minsky, I. Rivin, P. Sarnak and S. P. Tan for numerous useful conversations.  Finally, we thank the anonymous referees for numerous helpful comments which improved the readability and completeness of the article.

The first named author is supported by Basic Science Research Program through the National Research Foundation of Korea (NRF) funded by the Ministry of Education, Science and Technology (2013R1A1A1058646). S. Kim is also supported by Samsung Science and Technology Foundation (SSTF-BA1301-06). The second named author is partially supported by NSF grant DMS-1203964.

\def\cprime{$'$}
\providecommand{\bysame}{\leavevmode\hbox to3em{\hrulefill}\thinspace}
\providecommand{\MR}{\relax\ifhmode\unskip\space\fi MR }
\providecommand{\MRhref}[2]{%
  \href{http://www.ams.org/mathscinet-getitem?mr=#1}{#2}
}
\providecommand{\bysame}{\leavevmode\hbox to3em{\hrulefill}\thinspace}
\providecommand{\MR}{\relax\ifhmode\unskip\space\fi MR }
\providecommand{\MRhref}[2]{%
  \href{http://www.ams.org/mathscinet-getitem?mr=#1}{#2}
}
\providecommand{\href}[2]{#2}

\end{document}